\documentclass{article}[12pt]
\usepackage{amsmath}
\usepackage{amsfonts}
\usepackage{amssymb}
\usepackage{appendix}
\usepackage{amsthm }
\usepackage{footnote}
\usepackage[numbers,colon]{natbib}

\usepackage{enumerate}
\newtheorem{Thm}{Theorem}[section]

\newtheorem{Lem}[Thm]{Lemma}

\newtheorem{Pro}[Thm]{Proposition}

\makeatletter
\def\blfootnote{\xdef\@thefnmark{}\@footnotetext}
\makeatother

\theoremstyle{definition}
\newtheorem{Def}[Thm]{Definition}
\newtheorem{Eg}[Thm]{Example}

\theoremstyle{remark}
\newtheorem{Rem}[Thm]{\bf{Remark}}

\newcommand{\ConvD}{\overset{d}{\rightarrow}}
\newcommand{\ConvFDD}{\overset{f.d.d.}{\longrightarrow}}

\newcommand{\Cov}{\mathrm{Cov}}
\newcommand{\Var}{\mathrm{Var}}
\newcommand{\E}{\mathbb{E}}
\newcommand{\mathbd}{\boldsymbol}

\title{Generalized Hermite processes, discrete chaos and limit theorems }
\author{Shuyang Bai, Murad S. Taqqu}
\date{\today}
\begin{document}
\maketitle

\begin{abstract}
\blfootnote{
\begin{flushleft}
\end{flushleft}
\textbf{2010 AMS Classification:} 60G18, 60F05\\
}
We introduce a broad class of self-similar processes $\{Z(t),t\ge 0\}$ called generalized Hermite processes. They have stationary increments, are defined on a Wiener chaos with Hurst index $H\in (1/2,1)$, and include Hermite processes as a special case. They are defined through a homogeneous kernel $g$, called  ``generalized Hermite kernel'', which replaces the product of power functions in the definition of Hermite processes. The generalized Hermite kernels $g$ can also be used to generate  long-range dependent stationary sequences forming a discrete chaos process $\{X(n)\}$. In addition, we  consider a fractionally-filtered version $Z^\beta(t)$ of $Z(t)$, which allows $H\in (0,1/2)$.  Corresponding non-central limit theorems are  established. We also give a multivariate limit theorem which mixes central and non-central limit theorems.
\end{abstract}

\section{Introduction}\label{Sec:intro}
A stochastic process $\{X(t),t\ge 0\}$ with finite variance taking values in $\mathbb{R}$ is said to be \emph{self-similar} if there is a constant called \emph{Hurst coefficient} $H>0$, such that for any scaling factor $a>0$, $X(at)\overset{f.d.d.}{=}a^H X(t)$, where  $\overset{f.d.d.}{=}$  means equality in finite-dimensional distributions. If a self-similar process $\{X(t),t\ge 0\}$  has also stationary increments, namely, if for any $h\ge 0$, $\{Y(t):=X(t+h)-X(t),t\ge 0\}$ is a stationary process, then we say that $\{X(t),t\ge 0\}$ is $H$-sssi. The natural range of $H$ is $(0,1)$, which implies $\E X(t)=0$ for all $t\ge 0$. We refer the reader to Chapter 3 of \citet{embrechts:maejima:2002:selfsimilar} for details.

The  fundamental theorem of Lamperti (\citet{lamperti:1962:semi}) states that $H$-sssi processes are the only possible limit laws of normalized partial sum  of stationary sequences, that is, if
$$
\frac{1}{A(N)}\sum_{n=1}^{[Nt]}X(n)\ConvFDD Y(t)
$$
 and $A(N)\rightarrow\infty$ as $N\rightarrow\infty$, where $\{X(n)\}$ is stationary, then $\{Y(t),t\ge 0\}$ has to be $H$-sssi for some $H>0$, and $A(N)$ has to be regularly varying with exponent $H$. The notation $\ConvFDD$ stands for convergence in finite-dimensional distributions (f.d.d.).

The best known  example of  Lamperti's fundamental theorem is when $\{X(n)\}$ is  i.i.d.\ or a \emph{short-range dependent} (SRD) sequence, then the limit $Y(t)$ is Brownian motion which is $\frac{1}{2}$-sssi. If $\{X(n)\}$ has  \emph{long-range dependence} (LRD), the limit $Y(t)$ is often $H$-sssi with $H>1/2$. The most typical $H$-sssi process is  fractional Brownian motion $B_H(t)$, but there are also  non-Gaussian processes, e.g, \emph{Hermite processes} (\citet{taqqu:1979:convergence}, \citet{dobrushin:major:1979:non}). The Hermite process of order $1$ is fractional Brownian motion, but when the order is greater than or equal to $2$, its law belongs to higher-order Wiener chaos  (see, e.g., \citet{peccati:taqqu:2011:wiener}) and is thus non-Gaussian.

The Hermite processes have attracted a lot of attention. The first-order Hermite process, namely fractional Brownian motion, has been studied intensively by numerous researchers since its popularization by \citet{mandelbrot:vanness:1968:fractional}, and we refer the reader to a recent monograph \citet{nourdin:2012:selected} and the references therein. The second-order Hermite process, namely the Rosenblatt process, is also investigated in a number of papers. Recent works include \citet{tudor:2008:analysis}, \citet{bardet:tudor:2010:wavelet}, \citet{veillette:taqqu:2012:properties}, \citet{maejima:tudor:2007:wiener,maejima:tudor:2013:distribution}.  Hermite processes  frequently appear in statistical inference problems involving LRD, e.g., \citet{levy:boistard:taqqu:reisen:2011:asymptotic}, \citet{dehling:rooch:2012:non}.

It is interesting to note that when the stationary sequence $\{X(n)\}$ is LRD, one can obtain in the limit a much richer class of processes, whereas  in the SRD case, one obtains only Brownian motion. The type of limit theorems involving $H$-sssi  processes other than Brownian motion are often called \emph{non-central limit theorems}. While Hermite processes are the main examples of $H$-sssi processes obtained as the limit of partial sum of finite-variance LRD sequence, there are very few other limit $H$-sssi processes which have been considered, with some exceptions \citet{rosenblatt:1979:some} and \citet{major:1981:limit}.

In this paper, we   introduce a broad class of $H$-sssi ($H>1/2$) processes $\{Z(t),t\ge 0\}$ with their laws in  Wiener chaos,  which includes the Hermite processes as a special case. These processes are defined as $Z(t)=I_k(h_t)$, where $I_k(\cdot)$ denotes $k$-tuple Wiener-It\^o integral, and
$$h_t(x_1,\ldots,x_k):=\int_0^t g(s-x_1,\ldots,s-x_k)\mathrm{1}_{\{s>x_1,\ldots,s>x_k\}}ds,$$ with $g$ being some suitable homogeneous function on $\mathbb{R}_+^k$  called \emph{generalized Hermite kernel}. For example,
\begin{equation}\label{eq:eg}
g(x_1,\ldots,x_k)=\max\left( \frac{x_1\ldots x_k}{ x_1^{k-\alpha}+\ldots +x_k^{k-\alpha}} ,~ x_1^{\alpha/k}\ldots x_k^{\alpha/k}\right), \quad \mathbf{x} \in \mathbb{R}_+^k, ~\alpha\in (-k/2-1/2,-k/2).
\end{equation}
We call the corresponding H-sssi process $Z(t)$ a \emph{generalized Hermite process}.
We then construct a class of \emph{discrete chaos processes}
as
$$X(n)=\sum_{(i_1,\ldots,i_k)\in\mathbb{Z}_+^k}' g(i_1,\ldots,i_k)\epsilon_{n-i_1}\ldots \epsilon_{n-i_k},$$
where $\{\epsilon_i\}$ are i.i.d.\ noise, and the prime $'$ exclusion of the diagonals $i_p=i_q$, $p\neq q$. We show that the normalized partial sum of $X(n)$  converges to the generalized Hermite process $Z(t)$ defined by the same $g$. We also obtain processes  with $H\in(0,1/2)$ by applying an additional fractional filter. The increments of these processes have negative dependence. Finally, we state a multivariate limit theorem which mixes central and non-central limits, including cases where there is an additional fractional filter.

The paper is organized as follows. In Section 2, we review the Hermite processes. In Section 3, the generalized Hermite processes are introduced. In Section 4, we consider the discrete chaos processes. In Section 5, we prove a hypercontractivity relation for infinite discrete chaos. In Section 6, we show that the discrete chaos processes converge weakly to the generalized Hermite processes, including  situations where $H<1/2$.

\section{Brief review of Hermite processes}\label{Sec:Review}

The Hermite processes are defined with the aid of a multiple stochastic integral called \emph{Wiener-It\^o integral}. We give here a brief introduction to this integral. For the proofs of our statements and additional details, we refer the reader to \citet{major:1981:multiple} and \citet{nualart:2006:malliavin}, for example. The Wiener-It\^o integral is defined for any $f\in L^2(\mathbb{R}^k)$ as
\[
I_k(f):=\int_{\mathbb{R}^k}' f(x_1,\ldots,x_k) W(dx_1)\ldots W(dx_k),
\]
where $W(\cdot)$ is  Brownian motion viewed as a random integrator, and the prime $'$ indicates that we don't integrate on  the diagonals $x_p=x_q$, $p\neq q$.
The integral $I_k(\cdot)$ can be defined first for elementary functions $f=\sum_{i=1}^n a_i \mathrm{1}_{A_i}$, where $A_i$'s are off-diagonal cubes in $\mathbb{R}^k$. This results in a  linear combination of $k$-fold product of independent centered Gaussian random variables. One then extends this in the usual way to any $f\in L^2(\mathbb{R}^k)$. The random variable  $I_k(f)$ is also said to belong to the $k$-th Wiener chaos $\mathcal{H}_k$, which is the Hilbert space generated by $I_k(f)$ when $f$ varies in $L^2(\mathbb{R}^k)$. Here we state the following important properties of the Wiener-It\^o integral $I_k(\cdot)$:
\begin{enumerate}
\item $I_k(\cdot)$ is a linear mapping from $L^2(\mathbb{R}^k)$ to $L^2(\Omega)$.
\item If $f_\sigma(x_1,\ldots,x_k):=f(x_{\sigma(1)},\ldots,x_{\sigma(k)})$, where $\sigma$ is any permutation of $(1,\ldots,k)$, then $I_k(f_\sigma)=I_k(f)$. It hence suffices to focus on symmetric integrands (symmetrize $f$ as
    $$
    \tilde{f}(x_1,\ldots,x_k):=\frac{1}{k!}\sum_{\sigma}f(x_{\sigma(1)},\ldots,x_{\sigma(k)})
    $$
    when necessary).
\item Suppose $f\in L^2(\mathbb{R}^p)$ and $g\in L^2(\mathbb{R}^q)$, and both are symmetric. Then
\[
\E I_p(f) I_q (g) =
\begin{cases}
k! \langle f,g\rangle_{L^2(\mathbb{R}^k)}=k!\int_{\mathbb{R}^k}f(\mathbf{x})g(\mathbf{x})d\mathbf{x}, & \text{ if } p=q=k;  \\
0, & \text{ if } p\neq q.
\end{cases}
\]
If $f\in L^2(\mathbb{R}^k)$ is not  symmetric, one gets
\[
\E I_p(f)^2 = \|\tilde{f}\|_{L^2(\mathbb{R}^k)}^2\le k! \|{f}\|_{L^2(\mathbb{R}^k)}^2.
\]
\end{enumerate}

An Hermite process of order $k$ is an $H$-sssi process with $1/2<H<1$, which is represented by the following Wiener-It\^o integral:
\begin{equation}\label{eq:Herm TimeDomain}
Z_H^{(k)}(t)=a_{k,d} \int'_{\mathbb{R}^k}~\int_0^t \prod_{p=1}^k (s-x_j)_+^{d-1}ds~ W(dx_1)\ldots W(dx_k),
\end{equation}
where

 and $a_{k,d}$ is some positive constant that makes $\Var(Z_H^{(k)}(1))=1$. We call (\ref{eq:Herm TimeDomain}) the \emph{time-domain representation}. It is known that Hermite processes admit  other  representations in terms of Wiener-It\^o integrals (see \citet{pipiras:taqqu:2010:regularization}), among which we note the \emph{spectral-domain representation}:
\begin{equation}\label{eq:Herm SpecDomain}
Z_H^{(k)}(t)=b_{k,d}\int''_{\mathbb{R}^k}\frac{e^{i(u_1+\ldots+u_k)t}-1}{i(u_1+\ldots+u_k)} |u_1|^{-d}\ldots |u_k|^{-d} \widehat{W}(du_1)\ldots\widehat{W}(du_k),
\end{equation}
where $\widehat{W}(\cdot)$ is a  complex-valued Brownian motion (with real and imaginary parts being independent) viewed as a random integrator (see, e.g., p.22 of \citet{embrechts:maejima:2002:selfsimilar}), the double prime $''$ indicates the exclusion of the hyper-diagonals $u_p=\pm u_q$, $p\neq q$, and $b_{k,d}$ is some  positive constant that makes $\Var(Z_H^{(k)}(1))=1$. In the sequel, we use $\widehat{I}_k(\cdot)$ to denote a $k$-tuple Wiener-It\^o integral  with respect to the complex-valued Brownian motion $\widehat{W}(\cdot)$. In fact, the kernel inside the  Wiener-It\^o integral in (\ref{eq:Herm SpecDomain}) is the Fourier  transform of the kernel in (\ref{eq:Herm TimeDomain}) up to some unimportant factors.   The connection between the time-domain and spectral-domain representation is through the following general result:
\begin{Pro}\label{Pro:Time<->Spec}(Proposition 9.3.1 of \citet{peccati:taqqu:2011:wiener})
Let $g_j(\mathbf{x})$ be a real-valued function in $L^2(\mathbb{R}^{k_j})$, $j=1,\ldots,J$. Let
$$\widehat{g}_j(\mathbf{u})=\int_{\mathbb{R}^k} g_j(\mathbf{x}) e^{i\langle \mathbf{u},\mathbf{x} \rangle} d\mathbf{x}
$$
be the Fourier transform. Then
\[
\Big(I_{k_1}(g_1),\ldots,I_{k_J}(g_2)\Big)\overset{d}{=}\left((2\pi)^{-k_1/2}\widehat{I}_{k_1}(\widehat{g}_1 w_1^{\otimes_{k_1}}),\ldots,(2\pi)^{-k_J/2}\widehat{I}_{k_J}(\widehat{g}_2 w_J^{\otimes_{k_J}})\right),
\]
for any $|w_j(u)|=1$ and $w_j(u)=\overline{w_j(-u)}$, $j=1,\ldots,J$, where $w^{\otimes k}(u_1\ldots u_k):=w(u_1)\ldots w(u_k)$.
\end{Pro}
The factors $w_j^{\otimes {k_j}}$, $j=1,\ldots,J$ do not change the distributions due to the change-of-variable formula of  Wiener-It\^o integrals (see, e.g., Proposition 4.2 of \citet{dobrushin:1979:gaussian}).

The Hermite process of order $k=1$ is fractional Brownian motion $B_H(t)$, and that of order $k=2$ is called \emph{Rosenblatt process} whose marginal distribution was  discovered by \citet{rosenblatt:1961:independence}. We note that all $H$-sssi processes with unit variance at $t=1$ have covariance
$$
R(s,t)=\frac{1}{2}(s^{2H}+t^{2H}-|s-t|^{2H}),
$$
  as is the case for Hermite process of arbitrary order.

Hermite processes arise as limits  of partial sum of nonlinear LRD sequences. In the following two theorems,   $A(N)$ is a normalization factor  guaranteeing unit asymptotic variance for the partial sum process at $t=1$. We use  $\Rightarrow$ to denote weak convergence in the Skorohod space $D[0,1]$ with the uniform metric.
\begin{Thm}\label{Thm:GaussSub}(\citet{dobrushin:major:1979:non,taqqu:1979:convergence}.)
Suppose that $\{X(n)\}$ is a Gaussian stationary sequence with autocovariance
$$
\gamma(n)\sim cn^{2d-1}
$$
 as $n\rightarrow\infty$ for some constant $c>0$ and
$$1/2(1-1/k)<d<1/2.$$
Let $H_k(x):=(-1)^ke^{x^2/2}\frac{d^k}{dx^k}e^{-x^2/2}$ be the $k$-th Hermite polynomial, $k\ge 1$. Then
\[
\frac{1}{A(N)}\sum_{n=1}^{[Nt]} H_k(X(n))\Rightarrow Z_{d}^{(k)}(t).
\]
\end{Thm}

\begin{Thm}\label{Thm:Polyform}(\citet{surgailis:1982:zones}, see also \citet{giraitis:koul:surgailis:2009:large} Chapter 4.8.)
Let $\{\epsilon_i\}$ be an i.i.d.\ sequence with mean $0$ variance $1$,
$$
a_n\sim c n^{d-1}
$$
 as $n\rightarrow\infty$ for some constant $c>0$ and
$$1/2(1-1/k)<d<1/2.$$
Let
$$
X(n)=\sum_{0<i_1,\ldots,i_k<\infty}^{\prime} a_{i_1}\ldots a_{i_k} \epsilon_{n-i_1}\ldots\epsilon_{n-i_k},
 $$
 where the prime $'$ indicates that one doesn't sum on the diagonals $i_p=i_q$ $p\neq q$. Then
\[
\frac{1}{A(N)}\sum_{n=1}^{[Nt]} X(n)\Rightarrow Z_{d}^{(k)}(t).
\]
\end{Thm}
\begin{Rem}
The Hermite polynomial in Theorem \ref{Thm:GaussSub} can be replaced by a general function $G(\cdot)$ such that $\E G(X_n)=0$, $\E G(X_n)^2<\infty$,   due to the orthogonal expansion of $G(x)$ with respect to Hermite polynomials, and the fact that only the leading term in the expansion contributes to the limit law. Similarly, the off-diagonal multilinear polynomial-form process $X(n)$ in Theorem \ref{Thm:Polyform} can be replaced by a suitable  function  of the linear process $Y(n):=\sum_{i\ge 1} a_i\epsilon_{n-i}$. In both of the above theorems $\ConvFDD$ can be strengthened to weak convergence $\Rightarrow$ (Proposition 4.4.2 of \citet{giraitis:koul:surgailis:2009:large}).
\end{Rem}

\begin{Rem}
The range of the parameter $d$ in both of the theorems guarantees that the summand is LRD in the sense that the autocovariance decays as a power funciton with an exponent in the range $(-1,0)$. We note also that the  constant $c>0$ appearing in both  theorems can be replaced by a slowly varying function.
\end{Rem}

\section{Generalized Hermite Processes}\label{Sec:GenHermProc}
We introduce  first some notation, which will be used throughout.  $\mathbb{R}_+=(0,\infty)$, $\mathbb{Z}_+=\{1,2,\ldots\}$. $\mathbf{x}=(x_1,\ldots,x_k)\in \mathbb{R}^k$,  $\mathbf{i}=(i_1,\ldots,i_k)\in \mathbb{Z}^k$, $\mathbf{0}=(0,\ldots,0)$, $\mathbf{1}=(1,\ldots,1)$. For any real number $x$, $[x]=\sup \{n\in \mathbb{Z},n\le x\}$, and $[\mathbf{x}]=([x_1],\ldots,[x_k])$. We write $\mathbf{x}> \mathbf{y}$ (or $\ge$) if $x_j> y_j$ (or $\ge$), $j=1,\ldots,k$. $\langle \mathbf{x},\mathbf{y}\rangle=\sum_{j=1}^k x_jy_j$, and $\|\mathbf{x}\|=\sqrt{\langle \mathbf{x},\mathbf{x}\rangle}$, while $\|\cdot\|$ with a subscript is also used to denote the norm of some other  space (specified in the subscript).  Given a set $A\subset \mathbb{R}$, $A^k$ is the $k$-fold Cartesian product. $\mathrm{1}_A(\cdot)$ is the indicator function of a set $A$.
$L^p(\mathbb{R}^k,\mu)$ denotes the $L^p$-space on $\mathbb{R}^k$ with measure $\mu$, and $\mu$ is omitted if it is Lebesgue measure.
\subsection{General kernels}\label{Subsec:General}
The following proposition provides a general way to construct in the time-domain an $H$-sssi process living in Wiener chaos:
\begin{Pro}\label{Pro:Construct H-sssi} Fix an $H\in(0,1)$.
Suppose that $\{h_t(\cdot), t>0\}$ is a family of  functions defined on $\mathbb{R}^k$ satisfying
\begin{enumerate}
\item $h_t\in L^2(\mathbb{R}^k)$\label{SSSI:L2};
\item $\forall\lambda>0$, $\exists \beta \neq 0$,  such that $h_{\lambda t}(\mathbf{x})=\lambda^{H+k\beta /2}h_t(\lambda^\beta \mathbf{x})$ for a.e.\ $\mathbf{x}\in \mathbb{R}^k$ and all $t>0$; \label{SSSI:homo}
\item $\forall s>0$, $\exists$ $\mathbf{a}\in \mathbb{R}^k$, such that $h_{t+s}(\mathbf{x})-h_{t}(\mathbf{x})=h_s(\mathbf{x}+t\mathbf{a})$ for a.e.\ $\mathbf{x}\in \mathbb{R}^k$ and all $t>0$.\label{SSSI:stationary}
\end{enumerate}
Then $Z(t):=I_k(h_t)$  is an $H$-sssi process.
\end{Pro}
Condition \ref{SSSI:L2} guarantees that the Wiener-It\^o integral is well defined.
Condition \ref{SSSI:homo} yields self-similarity, where the term $k\beta /2$ in the exponent compensates for the scaling of the $k$-tuple Brownian motion integrators. Condition \ref{SSSI:stationary} guarantees  stationary increments. Self-similarity and stationary increments can  be rigorously checked by the change-of-variable formula of  Wiener-It\^o integrals (Proposition 4.2 of \citet{dobrushin:1979:gaussian}).

The Hermite process, for instance, which is defined in (\ref{eq:Herm TimeDomain}) can be obtained following the scheme of  Proposition \ref{Pro:Construct H-sssi} by letting
$$
h_t(\mathbf{x})=\int_0^t g(s\mathbf{1}-\mathbf{x})\mathrm{1}_{\{s\mathbf{1}>\mathbf{x}\}}(s)ds,
$$
 and
\begin{align}\label{eq:original g}
g(\mathbf{x})=\prod_{j=1}^k x_j^{d-1}, ~x_j>0.
\end{align}
It is easy to check that the conditions on $h_t$ in Proposition \ref{Pro:Construct H-sssi} are all satisfied with $\beta=-1$ in condition \ref{SSSI:homo} and $H=kd-k/2+1$.
One can also check that the integrand  in the spectral-domain representation in (\ref{eq:Herm SpecDomain}) also satisfies the first two conditions in Proposition \ref{Pro:Construct H-sssi}, but with $\beta=1$ in Condition \ref{SSSI:homo} instead. The third condition, however, must be replaced by  $\widehat{h}_{t+s}(\mathbf{u})-\widehat{h}_{t}(\mathbf{u})=e^{-it\langle\mathbf{a},\mathbf{u} \rangle} \widehat{h}_s(\mathbf{u})$ due to the Fourier-transform relation.

Our first  goal is to extend the kernel $g$ in (\ref{eq:original g}) to some  general class of functions. To do so, we define the following class of functions on $\mathbb{R}_+^k$, which first appeared in \citet{mori:toshio:1986:law} to study the law of iterated logarithm:
\begin{Def}\label{Def:GHK}
We say that a nonzero measurable function  $g(\mathbf{x})$ defined on $\mathbb{R}_+^k$ is a \emph{generalized Hermite kernel}, if it satisfies
\begin{enumerate}[A.]
\item $g(\lambda \mathbf{x})=\lambda^\alpha g(\mathbf{x})$, $\forall \lambda>0$, where $\alpha\in(-\frac{k+1}{2},-\frac{k}{2})$;\label{ass:homo}
\item $\int_{\mathbb{R}_+^k}|g(\mathbf{x})g(\mathbf{1}+\mathbf{x})| d\mathbf{x}  <\infty$. \label{ass:int 2}
\end{enumerate}
\end{Def}
One can check that the Hermite kernel $g$ in (\ref{eq:original g}) satisfies the above assumptions.
\begin{Rem}
The range of $\alpha$ in Condition \ref{ass:homo} is non-overlapping for different $k$, and extends from $-1/2$ to $-\infty$ with all the multiples of $-1/2$ excluded.
\end{Rem}

\begin{Rem}
Suppose $g_1$ and $g_2$ are generalized Hermite kernels having order $k_1$, $k_2$ and  homogeneity exponent $\alpha_1$, $\alpha_2$ respectively. If in addition, $\alpha_1+\alpha_2>-(k_1+k_2+1)/2$, then $g_1\otimes g_2(\mathbf{x}_1,\mathbf{x}_1):=g_1(\mathbf{x}_1)g_2(\mathbf{x}_2)$ is a generalized Hermite kernel having order $k_1+k_2$ and  homogeneity exponent $\alpha_1+\alpha_2$.
\end{Rem}
\begin{Thm}\label{Thm:is H-sssi}
Let $g(\mathbf{x})$ be a generalized Hermite kernel defined in Definition \ref{Def:GHK}. Then
$$h_t(\mathbf{x})=\int_0^t g(s\mathbf{1}-\mathbf{x})\mathrm{1}_{\{s\mathbf{1}>\mathbf{x}\}} ds$$
is well-defined in $L^2(\mathbb{R}^k)$, $\forall t>0$, and the process defined by $Z_t:=I_k(h_t)$ is an $H$-sssi process with
$$H=\alpha+k/2+1\in (1/2,1).$$
\end{Thm}
\begin{proof}
To check that $h_t\in L^2(\mathbb{R}^k)$, we write
\begin{align*}
\int_{\mathbb{R}^k} h_t(\mathbf{x})^2 d\mathbf{x}
&= \int_{\mathbb{R}^k} d\mathbf{x} \int_0^t\int_0^t ds_1ds_2~ g(s_1\mathbf{1}-\mathbf{x}) g(s_2\mathbf{1}-\mathbf{x})\mathrm{1}_{\{s_1\mathbf{1}>\mathbf{x}\}} \mathrm{1}_{\{s_2\mathbf{1}>\mathbf{x}\}}.
\end{align*}
We want to change the integration order by integrating on $\mathbf{x}$ first. By Fubini, we need to check that the absolute value of the integrand is integrable, that is,
\begin{align*}
&2\int_0^tds_1 \int_{s_1}^t ds_2 \int_{\mathbb{R}^k} d\mathbf{x}~ |g(s_1\mathbf{1}-\mathbf{x})g(s_2\mathbf{1}-\mathbf{x})|\mathrm{1}_{\{s_1\mathbf{1}-\mathbf{x}>\mathbf{0} \}} \quad (\text{ by symmetry of } s_1<s_2 \text{ and } s_1>s_2)
 \\
&=2\int_0^t ds \int_0^{t-s} du \int_{\mathbb{R}_+^k}d\mathbf{w}~
 |g(\mathbf{w})g(u\mathbf{1}+\mathbf{w})|  ~~\qquad\qquad\qquad (s=s_1, ~u=s_2-s_1,~\mathbf{w}=s_1\mathbf{1}-\mathbf{x})\\
&=2\int_0^t ds \int_0^{t-s} du \int_{\mathbb{R}_+^k}u^kd\mathbf{y}~
 |g(u\mathbf{y})g(u+u\mathbf{y})|  \\
&=2\int_0^t ds \int_0^{t-s} u^{2\alpha+k}du~\int_{\mathbb{R}_+^k}d\mathbf{y}~
 |g(\mathbf{y})g(\mathbf{1}+\mathbf{y})|  ~~~\qquad\qquad (\text{by Condition \ref{ass:homo} of Definition \ref{Def:GHK}}),
\end{align*}
where the last expression is finite by $2\alpha+k+1>0$  and Condition \ref{ass:int 2}. Hence by the same calculation, but without absolute values,
\begin{align*}
\int_{\mathbb{R}^k} h_t(\mathbf{x})^2 d\mathbf{x}
&=2\int_0^t ds \int_0^{t-s} u^{2\alpha+k}du~~ \int_{\mathbb{R}_+^k}d\mathbf{y}~
 g(\mathbf{y})g(\mathbf{1}+\mathbf{y}) \\
&=\frac{t^{2\alpha+k+2}}{(\alpha+k/2+1)(2\alpha+k+2)}\int_{\mathbb{R}_+^k}d\mathbf{y}~
 g(\mathbf{y})g(\mathbf{1}+\mathbf{y}).
\end{align*}

To check self-similarity (Condition \ref{SSSI:homo} of Proposition \ref{Pro:Construct H-sssi} with $\beta=-1$),
\begin{align*}
h_{\lambda t}( \mathbf{x})&=\int_0^{\lambda t} g(s\mathbf{1}-\mathbf{x})\mathrm{1}_{\{s\mathbf{1}>\mathbf{x}\}}ds=\lambda^{\alpha+1}\int_0^{t} g( r\mathbf{1}-\lambda^{-1}\mathbf{x})\mathrm{1}_{\{ r\mathbf{1}>\lambda^{-1}\mathbf{x}\}}\lambda dr
=\lambda^{\alpha+1}h_t(\lambda^{-1} \mathbf{x}),
\end{align*}
where the second equality uses Condition \ref{ass:homo} of Definition \ref{Def:GHK}. The Hurst coefficient $H$ of $I_k(h_t)$ is obtained from $\alpha+1=H-k/2$.
To check stationary increments
(Condition \ref{SSSI:stationary} of Proposition \ref{Pro:Construct H-sssi}), for any $t,r>0$,
\begin{align*}
h_{t+r}(\mathbf{x})-h_t(\mathbf{x})=\int_t^{t+r} g(s\mathbf{1}-\mathbf{x})\mathrm{1}_{\{s\mathbf{1}>\mathbf{x}\}}ds=\int_0^{r} g(u\mathbf{1}+t\mathbf{1}-\mathbf{x})\mathrm{1}_{\{u\mathbf{1}+t\mathbf{1}>\mathbf{x}\}} du=h_{r}(\mathbf{x}-t\mathbf{1}).
\end{align*}
\end{proof}

\begin{Rem}\label{Rem:byproduct}
As a byproduct of the above proof, we obtain that under the conditions  of Definition \ref{Def:GHK}, one has $\int_0^t |g(s\mathbf{1}-\mathbf{x})|\mathrm{1}_{\{s\mathbf{1}>\mathbf{x}\}}(s) ds<\infty$ for a.e.\ $\mathbf{x}\in \mathbb{R}^k$, and
$$
\E Z(t)^2 (k!)^{-1}\le \|h_t\|_{L^2(\mathbb{R}^k)}^2=\frac{t^{2H}}{H(2H-1)} C_g,
$$
where $C_g:=\int_{\mathbb{R}_+^k}g(\mathbf{x})g(\mathbf{1}+\mathbf{x})d\mathbf{x}$, and the first inequality becomes equality if $g$ and hence $h_t$ is symmetric. Note that $C_g>0$ must hold, otherwise $h_t(\mathbf{x})=\int_0^t g(s\mathbf{1}-\mathbf{x})\mathrm{1}_{\{s\mathbf{1}>\mathbf{x}\}} ds=0$ for a.e.\ $\mathbf{x}\in \mathbb{R}^k$ and any $t>0$, which implies that $g$ is  zero a.e., and thus contradicts  the assumption.
\end{Rem}

\begin{Rem}\label{Rem:symmetrization}
Since $\forall f\in L^2(\mathbb{R}^k)$, $I_k(f)=I_k(\tilde{f})$, where $\tilde{f}$ is the symmetrization of $f$ (\citet{nualart:2006:malliavin} p.9), it suffices to focus  on symmetric generalized Hermite kernels $g$ only. In the sequel, we will not always assume that $g$ is symmetric for convenience, while being aware that $g$ can always be symmetrized.
\end{Rem}
\begin{Def}\label{Def:GHP}
The process
\begin{equation}\label{eq:Z(t)}
Z(t):=\int_{\mathbb{R}^k}' ~\int_0^t g(s-x_1,\ldots,s-x_k) \mathrm{1}_{\{s>x_1,\ldots,s>x_k\}}ds~W(dx_1)\ldots W(dx_k)
\end{equation}
which we simply write
$Z(t)=I_k(h_t)$ with $h_t(\mathbf{x})=\int_0^t g(s\mathbf{1}-\mathbf{x})\mathrm{1}_{\{s\mathbf{1}>\mathbf{x}\}}ds$, where $g$ is a generalized Hermite kernel defined in Definition \ref{Def:GHK}, is called a \emph{generalized Hermite process}.
\end{Def}

\begin{Rem}
It is known (see, e.g., \citet{janson:1997:gaussian} Theorem 6.12) that if a  random variable $X$ belongs to the $k$-th Wiener chaos, then there $\exists  a,b,t_0>0$ such that for $t\ge t_0$, $$\exp(-at^{2/k})\le P(|X|>t) \le \exp(-bt^{2/k}).$$
This shows that the generalized Hermite processes of different orders must necessarily have different laws, and the higher the order gets, the heavier the tail of the marginal distribution becomes, while they all have moments of any order.
\end{Rem}

The generalized Hermite process $Z(t)$ admits a continuous version, which follows from the following general result:
\begin{Pro}
If $\{Z(t),t\ge 0\}$ is an $H$-sssi process whose marginal distribution satisfies $\E|Z(1)|^\gamma<\infty$ for some $\gamma>H^{-1}$, then $Z(t)$ admits a continuous version.
\end{Pro}
\begin{proof}
Using stationary increments and self-similarity, we have
$$
\E|Z(t)-Z(s)|^\gamma =\E|Z(t-s)|^\gamma=|t-s|^{H\gamma}\E|Z(1)|^\gamma.
 $$
 Since $H\gamma>1$,   Kolmogorov's criterion applies.
\end{proof}
\begin{Rem}
In \citet{mori:toshio:1986:law}, the following laws of iterated logarithm
are obtained for the generalized Hermite process $Z(t)$:
\begin{align*}
\limsup_{n\rightarrow\infty} \frac{Z(n)}{n^H (2\log_2 n)^{k/2}}=l_1, \quad
\liminf_{n\rightarrow\infty} \frac{Z(n)}{n^H (2\log_2 n)^{k/2}}=l_2 ~\text{ a.s.},
\end{align*}
where $l_1=\sup K_h$ and $l_2=\inf K_h$ with the set
\[
K_h:=\left\{ \int_{\mathbb{R}^k}h_1(\mathbf{x})\xi(x_1)\ldots \xi(x_k)d\mathbf{x}:\|\xi\|_{L^2(\mathbb{R})}\le 1\right\}.
\]
\end{Rem}

In the spirit of  (\ref{eq:Herm SpecDomain}), we can consider the spectral-domain representation of the generalized Hermite processes. Since  $h_t(\mathbf{x})=\int_0^t g(s\mathbf{1}-\mathbf{x})\mathrm{1}_{\{s\mathbf{1}>\mathbf{x}\}}(s) ds\in L^2(\mathbb{R})$, it always has an $L^2$-sense Fourier transform  $\widehat{h}_t$.
We give  an explicit way to calculate $\widehat{h}_t$ when $g$ is integrable in a neighborhood of the origin. Note that since $g$ is homogeneous, it suffices to assume integrability on the unit cube $(0,1]^k$.
\begin{Pro}\label{Pro:ComputeFT}
Suppose that
\begin{equation}
\int_{(0,1]^k}|g(\mathbf{x})|<\infty \label{ass:int 1}.
\end{equation}
Let $g_n(\mathbf{x})=g(\mathbf{x})\mathrm{1}_{(0,n]^k}(\mathbf{x})$, and $\widehat{g}_n(\mathbf{u}):=\int_{\mathbb{R}^k}g_n(\mathbf{x})e^{i\langle \mathbf{u},\mathbf{x} \rangle}d\mathbf{x}$ be its Fourier transform. Set
$$
\widehat{h}_{t,n}:=\frac{e^{it\langle\mathbf{u},\mathbf{1} \rangle}-1}{i\langle\mathbf{u},\mathbf{1} \rangle}\widehat{g}_n(-\mathbf{u}),
$$
then $\widehat{h}_{t,n}$ converges in $L^2(\mathbb{R}^k)$ to $\widehat{h}_t$. Moreover, there is a function $\widehat{g}(\mathbf{u})$ defined for a.e.\ $\mathbf{u}\in \mathbb{R}^k$, such that,
\begin{equation}\label{eq:general spec}
\widehat{h}_t(\mathbf{u})=\frac{e^{it\langle\mathbf{u},\mathbf{1} \rangle}-1}{i\langle\mathbf{u},\mathbf{1} \rangle}\widehat{g}(-\mathbf{u}).
\end{equation}
\end{Pro}
\begin{proof}
 Due to (\ref{ass:int 1}), the Fourier transform of $g_n$ is well-defined pointwise  as
\begin{equation}\label{eq:hatg_n}
\widehat{g}_n(\mathbf{u})=\int_{\mathbb{R}^k}g(\mathbf{x})\mathrm{1}_{(0,n]^k}(\mathbf{x})e^{i\langle \mathbf{u},\mathbf{x} \rangle}d\mathbf{x}.
\end{equation}
Let
$$
h_{t,n}(\mathbf{x})=\int_0^t g_n(s\mathbf{1}-\mathbf{x})\mathrm{1}_{\{s\mathbf{1}>\mathbf{x}\}}(s) ds=\int_0^t g(s\mathbf{1}-\mathbf{x})\mathrm{1}_{\{\mathbf{x}<s\mathbf{1}\le \mathbf{x}+n\mathbf{1}\}}(s) ds.
 $$
 Note that $|g_n(\mathbf{x})|\le |g(\mathbf{x})|$, so  by the proof of Theorem \ref{Thm:is H-sssi}, $h_{t,n}(\mathbf{x})\in L^2(\mathbb{R}^k)$, and  by the Dominated Convergence Theorem, $h_{t,n}$ converges to $h_{t}$ pointwise as $n\rightarrow\infty$.  Since $|h_{t,n}|\le \int_0^t |g(s\mathbf{1}-\mathbf{x})|\mathrm{1}_{\{s\mathbf{1}>\mathbf{x}\}}(s) ds$, by  the Dominated Convergence Theorem in $L^2(\mathbb{R}^k)$,  $h_{t,n}$ converges to $h_{t}$ in $L^2(\mathbb{R}^k)$. By Plancherel's isometry,  $\widehat{h}_{t,n}$, the  Fourier transform of $h_{t,n}$, converges in $L^2(\mathbb{R}^k)$ to $\widehat{h}_t$.
But
\begin{align}
\widehat{h}_{t,n}(\mathbf{u}):=&\int_{\mathbb{R}^k} \int_{0}^t g(s\mathbf{1}-\mathbf{x}) \mathrm{1}_{\{\mathbf{x}<s\mathbf{1}\le\mathbf{x}+ n\mathbf{1}\}}(s)ds~ e^{i\langle \mathbf{u},\mathbf{x}\rangle}d\mathbf{x}\notag\\
=& \int_{0}^t \int_{\mathbb{R}^k} e^{i\langle\mathbf{u},s\mathbf{1}\rangle} g(s\mathbf{1}-\mathbf{x})e^{i\langle -\mathbf{u},s\mathbf{1}-\mathbf{x}\rangle}\mathrm{1}_{\{\mathbf{0}< s\mathbf{1}-\mathbf{x}\le n\mathbf{1}\}}(\mathbf{x}) d\mathbf{x} ds\notag\\
=&\int_{0}^t e^{i\langle\mathbf{u},s\mathbf{1}\rangle} ds\int_{\mathbb{R}^k}  g(\mathbf{y})\mathrm{1}_{\{\mathbf{0}<\mathbf{y}\le n\mathbf{1}\}}e^{i\langle -\mathbf{u},\mathbf{y}\rangle} d\mathbf{y} \notag\\
=& \frac{e^{it\langle\mathbf{u},\mathbf{1} \rangle}-1}{i\langle\mathbf{u},\mathbf{1} \rangle}\widehat{g}_n(-\mathbf{u}),\label{eq:compute fourier h_tn}
\end{align}
where the change of integration order is valid because by (\ref{ass:int 1}),
$$\int_0^tds \int_{\mathbb{R}^k}d\mathbf{x}|g(s\mathbf{1}-\mathbf{x})| \mathrm{1}_{\{\mathbf{x}<s\mathbf{1}\le\mathbf{x}+ n\mathbf{1}\}}=\int_0^tds \int_{\mathbb{R}^k}|g(\mathbf{y})|\mathrm{1}_{\{\mathbf{0}<\mathbf{y}\le n\mathbf{1}\}} d\mathbf{y}<\infty.$$

We now prove (\ref{eq:general spec}). The fact that $\widehat{h}_{t,n}$ converges in $L^2(\mathbb{R}^k)$ to $\widehat{h}_t$  implies that $\widehat{g}_n$ is a Cauchy sequence in $L^2(\mathbb{R}^k,\mu_t)$, where $\mu_t$ is the measure given by $$\mu_t(A)=\int_A \left|\frac{e^{it\langle\mathbf{u},\mathbf{1} \rangle}-1}{i\langle\mathbf{u},\mathbf{1} \rangle}\right|^2d\mathbf{u}=\int_A\frac{2-2\cos(t \langle\mathbf{u},\mathbf{1}\rangle)}{\langle\mathbf{u},\mathbf{1}\rangle^2}d\mathbf{u} $$
for any measurable set $A\subset \mathbb{R}^k$. Hence there exists a $\widehat{g} \in L^2(\mathbb{R}^k,\mu_t)$ which is the limit of $\widehat{g}_n$  in $L^2(\mathbb{R}^k,\mu_t)$. Since $\mu_t$ is equivalent to Lebesgue measure, $\widehat{g}$ is determined a.e.\ on $\mathbb{R}^k$, and there exists a subsequence of $\widehat{g}_n$ that converges a.e.\  to $\widehat{g}$. So (\ref{eq:general spec}) holds.

\end{proof}
\begin{Rem}
Note that $\widehat{g}$ is not the $L^2$-sense Fourier transform of $g\mathrm{1}_{\mathbb{R}^k_+}$, since $g\notin L^2{(\mathbb{R}^k_+)}$. One can, however,  evaluate the limit of $\widehat{g}_n$ pointwise as an improper  integral, as is done in the Hermite kernel case (\ref{eq:original g}) (see Lemma 6.2 of \citet{taqqu:1979:convergence}).
\end{Rem}
The limit $\widehat{g}$ in (\ref{eq:general spec}) is also a homogeneous function:
\begin{Pro}\label{Pro:hat g homo}
The function $\widehat{g}$ defined in Remark \ref{Pro:ComputeFT} satisfies  for any $\lambda>0$,
$g(\lambda \mathbf{u})=\lambda^{-\alpha-k}\widehat{g}(\mathbf{u})$ for a.e.\ $\mathbf{u}\in \mathbb{R}^k$.
\end{Pro}
\begin{proof}
Following (\ref{eq:hatg_n}) and using  Condition \ref{ass:homo} of Definition \ref{Def:GHK}, and noting that $\langle \lambda\mathbf{u},\mathbf{x}\rangle=\langle \mathbf{u},\lambda\mathbf{x}\rangle$,  we have
\begin{align*}
\widehat{g}_n(\lambda\mathbf{u})&=\lambda^{-\alpha}\int_{\mathbb{R}^k}g(\lambda\mathbf{x})\mathrm{1}_{(0,n]^k}(\mathbf{x})e^{i\langle \mathbf{u},\lambda\mathbf{x} \rangle}d\mathbf{x}
=\lambda^{-\alpha-k}\int_{\mathbb{R}^k}g(\mathbf{y})\mathrm{1}_{(0,\lambda n]^k}(\mathbf{\mathbf{y}})e^{i\langle \mathbf{u},\mathbf{y} \rangle}d\mathbf{y}
=\lambda^{-\alpha-k}\widehat{g}_{n\lambda }(\mathbf{u}).
\end{align*}
Then let $n\rightarrow\infty$ through a subsequence so that both sides converge a.e..
\end{proof}

\begin{Rem}
 The spectral-domain representation of the Hermite process in (\ref{eq:Herm SpecDomain}) is indeed obtained as $\widehat{g}(\mathbf{u})=c\prod_{j=1}^k |u_k|^{-d}w(\mathbf{u})$ for some constant $c>0$, where the function $w(\mathbf{u})=\prod_{j=1}^k \exp\left(-\mathrm{sign}(u_j)\frac{i\pi d}{2}\right)$ can be omitted (see Proposition \ref{Pro:Time<->Spec}).
\end{Rem}

\subsection{Special kernels and examples}\label{Subsec:special}
We  introduce now some subclasses of the generalized Hermite kernels $g$ defined in Definition \ref{Def:GHK}, which will  be of interest later when dealing with limit theorems. Note that the kernel $g$ is determined by its value on the positive unit sphere $\mathcal{S}_+^k:=\{\mathbf{x}\in \mathbb{R}_+^k, \|\mathbf{x}\|= 1\}$.  Because it is homogeneous, $g$ is always radially continuous and it is decreasing since $\alpha<0$ in Definition \ref{Def:GHK}. Thus assuming  that $g$ is  continuous  on $\mathcal{S}_+^k$  a.e.\ (with respect to the uniform measure on the $\mathcal{S}_+^k$) is the same as assuming  $g$ is continuous a.e.\ on $\mathbb{R}_+^k$ .
\begin{Def}\label{Def:class bounded}
We say that a generalized Hermite kernel $g$ is of Class (B) (B stands for ``boundedness''), if on ${\mathcal{S}_+^k}$,  it is continuous  a.e.  and bounded. Consequently,
$$
|g(\mathbf{x})| \le \|\mathbf{x}\|^\alpha g(\mathbf{x}/\|\mathbf{x}\|)\le c \|\mathbf{x}\|^\alpha
 $$
 for some $c>0$.
\end{Def}

\begin{Rem}
According to Lemma 7.1 of \citet{mori:toshio:1986:law}, Class (B) forms a dense subclass of the class of generalized Hermite kernels in the sense that for any generalized Hermite kernel $g$ and any $\epsilon>0$, there exists $g_\epsilon$ in Class (B), such that $\|h-h_\epsilon\|_{L^2(\mathbb{R}^k)}<\epsilon$, where $h(\mathbf{x})=\int_0^1 g(s\mathbf{1}-\mathbf{x})\mathrm{1}_{\{s\mathbf{1}>\mathbf{x}\}}ds$  and $h_\epsilon(\mathbf{x})=\int_0^1 g_\epsilon(s\mathbf{1}-\mathbf{x})\mathrm{1}_{\{s\mathbf{1}>\mathbf{x}\}}ds$.
\end{Rem}

Note that Class (B) does not include the original Hermite kernel in (\ref{eq:original g}). We now introduce a class of generalized Hermite kernels, called Class (L), which includes generalized Hermite kernels of the form:
\begin{equation}\label{eq:nonsym Herm}
g(\mathbf{x})=\prod_{j=1}^k x_j^{\gamma_j},
\end{equation}
where each $-1<\gamma_j<-1/2$ and $-k/2-1/2<\sum_j  \gamma_j < -k/2$. These particular kernels with $k=2$ has been considered in \citet{maejima:tudor:2012:selfsimilar} where the resulting process is called  non-symmetric Rosenblatt process. We  hence call the kernel in (\ref{eq:nonsym Herm}) a \emph{non-symmetric Hermite kernel}. Note that despite the name, one can always symmetrize these kernels.
Class (L) will appear in the discrete chaos processes and the limit theorems considered later.
\begin{Def}\label{Def:class limit}
We say that a generalized Hermite kernel $g$ on $\mathbb{R}_+^k$ having homogeneity exponent $\alpha$ is of Class (L) ($L$ stands for ``limit'' as in ``limit theorems''), if
\begin{enumerate}
\item $g$ is continuous a.e.\ on $\mathbb{R}^k_+$;
\item $|g(\mathbf{x})|\le  g^*(\mathbf{x})$ a.e.\ $\mathbf{x}\in \mathbb{R}_+^k$, where $g^*$ is a finite linear combination of non-symmetric Hermite kernels: $\prod_{j=1}^k x_j^{\gamma_j}$, where $\gamma_j\in (-1,-1/2)$, $j=1,\ldots,k$, and $\sum_{j=1}^k \gamma_j=\alpha\in(-k/2-1/2,-k/2)$.
\end{enumerate}
\end{Def}
For example, $g^*(\mathbf{x})$  could be $x_1^{-3/4} x_2^{-5/8}+x_1^{-9/16}x_2^{-13/16}$ if $k=2$. In this case, $\alpha=-11/8$.
\begin{Rem}\label{Rem:L good}
If two functions $g_1$ and $g_2$ on $\mathbb{R}_+^k$ satisfy Condition 2 of Definition \ref{Def:class limit},  then  $\int_{\mathbb{R}_+^k}|g_1(\mathbf{x})g_2(\mathbf{1}+\mathbf{x})|d\mathbf{x}<\infty$ automatically holds, which can be seen by using the following identity: for any $\gamma,\delta\in (-1,-1/2)$,
$$\int_0^\infty x^\gamma(1+x)^\delta dx=\mathrm{B}(\gamma+1,-\gamma-\delta-1),$$
 where $\mathrm{B}(\cdot,\cdot)$ is the beta function.
In addition, $\int_{(0,1]^k}|g_1(\mathbf{x})|d\mathbf{x}<\infty$ also holds.
\end{Rem}

\begin{Pro}\label{Pro:L > B}
Class (L) contains  Class (B).
\end{Pro}
\begin{proof}
Suppose $g$ is a generalized Hermite kernel of Class (B). Then there exist contants $C_1,C_2>0$, such that
\[
|g(\mathbf{x})|\le C_1 \|\mathbf{x}\|^{\alpha}=C_1\left(\sum_{j=1}^k x_j^2\right)^{\alpha /2}\le C_2 \prod_{j=1}^k x_j^{\alpha/k},
\]
where we have used the arithmetic-geometric mean inequality $k^{-1}\sum_{j=1}^ky_j\ge \left(\prod_{j=1}^k y_j\right)^{1/k}$ and $\alpha<0$. So Condition 2 of Definition \ref{Def:class limit} is satisfied with $g^*$ being a single term where $\gamma_1=\ldots=\gamma_k=\alpha/k$.
\end{proof}

\begin{Rem}
In view of Remark \ref{Rem:byproduct} and Remark \ref{Rem:L good}, one can check that Class (B) or Class (L) if adding in the a.e. $0$-valued function, with fixed order $k$ and fixed homogeneity component $\alpha\in (-k/2-1/2,-k/2)$, forms an inner product space, with the inner product specified as
\begin{align*}
\langle g_1,g_2\rangle &:=\left\langle\int_0^1 g_1(s\mathbf{1}-\cdot)ds ,\int_0^1 g_2(s\mathbf{1}-\cdot)ds \right\rangle_{L^2(\mathbb{R}^k)} \\
&=\frac{1}{2H(2H-1)}\int_{\mathbb{R}_+^k}g_1(\mathbf{x})g_2(\mathbf{1}+\mathbf{x})+g_1(\mathbf{1}+\mathbf{x})g_2(\mathbf{x})d\mathbf{x},
\end{align*}
where $H=\alpha+k/2+1$, which yields the norm
$$\|g\|:=\left\|\int_0^1 g(s\mathbf{1}-\cdot)ds\right\|_{L^2(\mathbb{R}^k)}=\left(\frac{1}{H(2H-1)}\int_{\mathbb{R}_+^k}g(\mathbf{x})g(\mathbf{1}+\mathbf{x})d\mathbf{x}\right)^{1/2}.
$$
\end{Rem}

Here are several examples.
\begin{Eg}
Suppose $g(\mathbf{x})=\|\mathbf{x}\|^\alpha$, where $\alpha\in(-1/2-k/2,-k/2)$. This $g$ belongs to Class (B) and thus also Class (L). The pseudo-Fourier transform (Proposition \ref{Pro:ComputeFT}) of $g$  is  $\widehat{g}(\mathbf{u})=c\|\mathbf{u}\|^{-\alpha-k}$ ((25.25) of \citet{samko:1993:fractional}) for some constant $c>0$, which provides the spectral representation by (\ref{eq:general spec}).
\end{Eg}

\begin{Eg}\label{Eg:not M}
 Another example of  Class (B):
$$g(\mathbf{x})=\frac{\prod_{j=1}^k x_j^{a_j}}{\sum_{j=1}^kx_j^b},$$
where $a_j>0$ and $b>0$, yielding  a homogeneity exponent  $\alpha=\sum_{j=1}^k a_j-b\in (-1/2-k/2,-k/2)$.
\end{Eg}

\begin{Eg}\label{eg:L}
We give yet another example of Class (L) but not (B):
$$
g(\mathbf{x})=g_0(\mathbf{x})\vee\left(\prod_{j=1}^k x_j^{\alpha/k}\right).
$$
where $g_0(\mathbf{x})>0$ is any generalized Hermite kernel of Class (B) on $\mathbb{R}_+^k$ with homogeneity exponent $\alpha$.
\end{Eg}

\subsection{Fractionally filtered kernels}\label{Subsec:frac}
According to Theorem \ref{Thm:is H-sssi}, the generalized Hermite process introduced above admits a Hurst coefficient $H>1/2$ only.  To obtain an $H$-sssi process with $0<H<1/2$, we consider the following fractionally filtered  kernel:
\begin{equation}\label{eq:h^beta_t}
h^\beta_t(\mathbf{x})=
\int_{\mathbb{R}}l_t^\beta(s) g(s\mathbf{1}-\mathbf{x})\mathrm{1}_{\{s\mathbf{1}>\mathbf{x}\}} ds,
\end{equation}
where $g$   is a generalized Hermite kernel defined in Definition \ref{Def:GHK} with homogeneity exponent $$
\alpha\in (-k/2-1/2,-k/2),
$$
and
\begin{equation}\label{eq:l^beta_t}
l^{\beta}_t(s)=\frac{1}{\beta}\left[(t-s)_+^\beta- (-s)_+^\beta\right],~\beta\neq 0.
\end{equation}
One can extend it to $\beta=0$ by writing $l^{0}_t(s)=\mathrm{1}_{(0,t]}(s)$, but this would lead us back to the generalized Hermite process case. We hence assume throughout that $\beta\neq 0$.
The following proposition gives the range of $\beta$ for which $I_k(h_t^\beta)$ is well-defined.
\begin{Pro}\label{Pro:beta range}
If
\begin{equation}\label{eq:beta range}
-1<-\alpha-\frac{k}{2}-1<\beta<-\alpha-\frac{k}{2}<\frac{1}{2}, \quad \beta\neq 0
\end{equation}
then $h^\beta_t \in L^2(\mathbb{R}^k)$.
\end{Pro}
\begin{proof}
\begin{align}
\int_{\mathbb{R}^k} h_t^\beta (\mathbf{x})^2 d\mathbf{x}\le&
2\int_{-\infty}^\infty ds_1 \int_{s_1}^\infty ds_2 \int_{\mathbb{R}^k} d\mathbf{x}~ l_t(s_1)l_t(s_2)|g(s_1\mathbf{1}-\mathbf{x})g(s_2\mathbf{1}-\mathbf{x})|\mathrm{1}_{\{ s_1\mathbf{1}>\mathbf{x}\}} \label{eq:int h^beta_t}
 \\
=&2\int_{-\infty}^\infty ds \int_{0}^\infty du \int_{\mathbb{R}_+^k}d\mathbf{w}~
l^{\beta}_t(s)  l^{\beta}_t(s+u) |g(\mathbf{w})g(u\mathbf{1}+\mathbf{w})|  \qquad\qquad (s=s_1, ~u=s_2-s_1,~\mathbf{w}=s_1\mathbf{1}-\mathbf{x})\notag\\
=&2\int_{-\infty}^\infty ds ~l^{\beta}_t(s) \int_{0}^\infty   l^{\beta}_t(s+u) u^{2\alpha+k}du~\int_{\mathbb{R}_+^k}d\mathbf{y}~
 |g(\mathbf{y})g(\mathbf{1}+\mathbf{y})|\notag.
\end{align}
We thus focus on showing $\int_{-\infty}^\infty ds ~l^{\beta}_t(s) \int_{0}^\infty   l^{\beta}_t(s+u) u^{2\alpha+k}du<\infty$. Recall that for any $c>0$, we have
\begin{align*}
\int_0^c (c-s)^{\gamma_1} s^{\gamma_2}ds= c^{\gamma_1+\gamma_2+1}\int_0^1 (1-s)^{\gamma_1}s^{\gamma_2}ds =c^{\gamma_1+\gamma_2+1}\mathrm{B}(\gamma_1+1,\gamma_2+1),\quad \forall\gamma_1,\gamma_2>-1.
\end{align*}
So by noting that $\beta>-1$ and $2\alpha+k>-1$, we have
\begin{align*}
\int_{0}^\infty   l^{\beta}_t(s+u) u^{2\alpha+k}du&=\frac{1}{\beta}\int_{0}^\infty \left[(t-s-u)_+^\beta -(-s-u)_+^\beta\right]u^{2\alpha+k} du\\
&= \frac{1}{\beta}\left[\int_0^{t-s}(t-s-u)^\beta u^{2\alpha+k}du+\int_0^{-s} (-s-u)^\beta u^{2\alpha+k}du\right]
\\&=
\frac{\mathrm{B}(\beta+1,2\alpha+k+1)}{\beta}\left[(t-s)_+^{\beta+\delta}- (-s)_+^{\beta+\delta}\right],
\end{align*}
where
\begin{equation}\label{eq:delta}
\delta =2\alpha+k+1\in (0,1).
\end{equation}
 We thus want to determine when the following holds:
\begin{align*}
\int_\mathbb{R} \left((t-s)_+^\beta - (-s)_+^\beta \right) \left((t-s)_+^{\beta+\delta} - (-s)_+^{\beta+\delta} \right)ds<\infty.
\end{align*}
Suppose $t>0$.
The potential integrability problems appear near $s=-\infty,0,t$. Near $s=-\infty$, the integrand behaves like $|s|^{2\beta+\delta-2}$, and thus we need $2\beta+\delta-2<-1$; near $s=0$, the integrand behaves like $|s|^{2\beta+\delta}$, and thus $2\beta+\delta>-1$; near $s=t$, the integrand behaves like $|t-s|^{2\beta+\delta}$, and thus again $2\beta+\delta>-1$. In view of (\ref{eq:delta}),  these requirements are satisfied by (\ref{eq:beta range}).
\end{proof}
\begin{Rem}
Using (\ref{eq:int h^beta_t}) we obtain as a byproduct of the preceding proof  that if $\beta$ is in the range given in Proposition \ref{Pro:beta range}, then the function $f_{\mathbf{x},t}(s):=l_t(s) |g(s\mathbf{1}-\mathbf{x})|\mathrm{1}_{\{s\mathbf{1}>\mathbf{x}\}}$ is in $L^1(\mathbb{R})$ for any $t>0$ and a.e.\ $\mathbf{x}\in \mathbb{R}^k$.
\end{Rem}
\begin{Thm}\label{Thm:Frac Z beta hsssi}
The process defined by $Z^\beta(t):=I_k(h_t^\beta)$ with $h_t^\beta$ given in (\ref{eq:h^beta_t}), namely,
\begin{equation}\label{eq:frac filt proc full}
Z^\beta(t)= \int_{\mathbb{R}^k}' \int_{\mathbb{R}}\frac{1}{\beta} [(t-s)_+^\beta-(-s)_+^\beta] g(s-x_1,\ldots,s-x_k)1_{\{s>x_1,\ldots,s>x_k\}}  ds~ W(dx_1)\ldots W(dx_k),
\end{equation}
 is an $H$-sssi process with
$$H=\alpha+\beta+k/2+1 \in (0,1).
$$
\end{Thm}
\begin{proof}
By (\ref{eq:l^beta_t}), one has for any $\lambda>0$,
$
l^\beta_{\lambda t}(s)=\lambda^{\beta} l^\beta_{t}(\frac{s}{\lambda}),
$
and for any $t,h>0$,
$
l^\beta_{t+h}(s)-l^\beta_{t}(s)=l^\beta_h(s-t).
$
In addition, $g$ is homogeneous with exponent $\alpha$. The conclusion then follows by Proposition \ref{Pro:Construct H-sssi}.
\end{proof}
\begin{Rem}
In the case $\beta>0$, one is able to write $l^\beta_t(s)=\int_0^t (r-s)_+^{\beta-1}dr$, and thus by Fubini
\begin{equation}\label{eq:h_t beta>0}
h_t^\beta (\mathbf{x})=\int_0^tdr \int_{\mathbb{R}} ds  (r-s)_+^{\beta-1}g(s\mathbf{1}-\mathbf{x})\mathrm{1}_{\{s\mathbf{1}>\mathbf{x}\}}.
\end{equation}
\end{Rem}
\begin{Rem}
To get the anti-persistent case $H<1/2$, choose
$$\beta\in(-\alpha-k/2-1,-\alpha-k/2-1/2).
$$
\end{Rem}

We now state an analog of (\ref{eq:general spec}) for the spectral representation of the process $Z^\beta(t)$:
\begin{Pro}\label{Pro:ComputeFT frac}
Suppose that (\ref{ass:int 1}) holds.
Then the $L^2$-sense Fourier transform of $h_t^\beta$ is
\begin{equation}\label{eq:hat h_t^beta}
\widehat{h}_{t}^\beta(\mathbf{u}) = (e^{it\langle\mathbf{u},\mathbf{1} \rangle}-1)(i\langle\mathbf{u},\mathbf{1} \rangle)^{-\beta-1} \widehat{g}(-\mathbf{u}) \Gamma(\beta),~a.e.~\mathbf{u}\in \mathbb{R}^k,
\end{equation}
where $\widehat{g}$ is defined in Proposition \ref{Pro:ComputeFT}.
\end{Pro}
\begin{proof}
Let  $g_n(\mathbf{x})=g(\mathbf{x})\mathrm{1}_{(0,n]^k}(\mathbf{x})$, and $l_{t,n}^\beta=\beta^{-1}[(t-s)_+^\beta \mathrm{1}_{\{t-s<n\}}-(-s)_+^\beta \mathrm{1}_{\{-s<n\}}].$ Set 
$$h_{t,n}^\beta(\mathbf{x}) = \int_{\mathbb{R}}l_{t,n}(s)g_n(s\mathbf{1}-\mathbf{x})ds.$$
Similar to the proof of Proposition \ref{Pro:ComputeFT}, one can  show that $h_{t,n}^\beta $ converges in $L^2(\mathbb{R}^k)$ to $h_{t}^\beta $ as $n\rightarrow\infty$ through the Dominated Convergence Theorem by noting that  $|g_n|\le |g|$ and $|l_{t,n}^\beta|\le l_{t}^\beta$.

Since the truncated $l_{t,n}$ and ${g}_n$ admit $L^1$-Fourier transforms $\widehat{l}_{t,n}$  and    $\widehat{g}_n$ respectively,  one can   write the Fourier transform of $h_{t,n}^\beta$ as:
\begin{align*}
\widehat{h}_{t,n}^\beta(\mathbf{u}) = \widehat{l}_{t,n}(\langle\mathbf{u},\mathbf{1}\rangle) \widehat{g}_n(-\mathbf{u}),
\end{align*}
(compare with (\ref{eq:compute fourier h_tn})). Since $h_{t,n}^\beta$ converges in $L^2(\mathbb{R})$ to $h_t^\beta$ as $n\rightarrow\infty$, by Plancherel's isometry, $\widehat{h}^\beta_{t,n}$ converges in $L^2(\mathbb{R}^k)$ to  $\widehat{h}^\beta_t$. One now needs to identify (\ref{eq:hat h_t^beta}) with the limit of $\widehat{h}^\beta_{t,n}$.

We first compute $\widehat{l}^\beta_{t,n}$. When $\beta<0$, one has  by change of variable that
\begin{align}
l^{\beta}_{t,n}(u)&=\beta ^{-1}\left(\int_\mathbb{R} e^{iux} (t-x)_+^\beta \mathrm{1}_{\{t-x<n\}}dx-\int_{\mathbb{R}}e^{iux}(-x)_+^\beta \mathrm{1}_{\{-x<n\}} dx\right)\notag\\
&= \beta^{-1}(e^{iut}-1)\int_{0}^n e^{-ius} s^\beta ds.\label{eq:hat l beta<0}
\end{align}
When $\beta>0$, one has
\begin{align*}
 l^{\beta}_{t,n}(u)&=\int_\mathbb{R}\mathrm{1}_{[0,t)}(x)(x-u)_+^{\beta-1}\mathrm{1}_{\{x-u<n\}}dx=(\mathrm{1}_{[0,t)} * b_n)(u),
\end{align*}
where $b_n(x)=(-x)_+^{\beta-1}\mathrm{1}_{\{-x<n\}}$. We have the Fourier transforms $\widehat{1_{[0,t)}}(u)=\frac{e^{iut}-1}{iu}$, and
\[
\widehat{b}_n(u)=\int_{\mathbb{R}}e^{-iux} (-x)_+^{\beta-1}\mathrm{1}_{\{-x<n\}}dx
=\int_0^n e^{-ius}s^{\beta-1}ds.
\]
So
\begin{equation}\label{eq:hat l beta>0}
\widehat{l}^\beta_{t,n}(u)=\frac{e^{iut}-1}{iu}\int_0^n e^{-ius}s^{\beta-1}ds
\end{equation}
By  \citet{gradshteyn:2007:table} Formula 3.761.4 and 3.761.9, for $\mu\in (0,1)$,
\begin{align*}
\lim_{n\rightarrow\infty}\int_0^n e^{-ius}s^{\mu-1}ds& = |u|^{-\mu}\Gamma(\mu)\cos(\frac{\mu \pi}{2})-i\mathrm{sign}(u)|u|^{-\mu}\Gamma(\mu)\sin(\frac{\mu\pi}{2})\\
&=e^{-i\mathrm{sign}(u)\mu\pi/2} |u|^{-\mu}\Gamma(\mu)=(iu)^{-\mu}\Gamma(\mu),
\end{align*}
Combining the foregoing limit with (\ref{eq:hat l beta<0}) and (\ref{eq:hat l beta>0}), we deduce
\[
\lim_{n\rightarrow\infty}\widehat{l}^\beta_{t,n}=\widehat{l}^\beta_{t}(u) :=(e^{itu}-1) (iu)^{-\beta-1} \Gamma(\beta).
\]
Recall that there exists a subsequence  $\widehat{g}_{n_k}$ converges a.e.\ to the pseudo-Fourier transform  $\widehat{g}$ as $k\rightarrow\infty$ (Proposition \ref{Pro:ComputeFT}). So $\widehat{l}_{t,n_k}(\langle\mathbf{u},\mathbf{1}\rangle) \widehat{g}_{n_k}(-\mathbf{u})$ converges to $\widehat{l}_{t}(\langle\mathbf{u},\mathbf{1}\rangle) \widehat{g}(-\mathbf{u})$ for a.e.\ $\mathbf{u}\in \mathbb{R}^k$. But at the same time $\widehat{l}_{t,n_k}(\langle\mathbf{u},\mathbf{1}\rangle) \widehat{g}_{n_k}(-\mathbf{u})$ converges in  $L^2(\mathbb{R})^k$ to $\widehat{h}_{t}^\beta$. So we identify $\widehat{h}_t^\beta$ with the expression in (\ref{eq:hat h_t^beta})
\end{proof}

\begin{Rem}
By Proposition \ref{Pro:Time<->Spec}, we get a spectral representation $Z^\beta(t)\overset{f.d.d.}{=}\widehat{I}(\widehat{h}^\beta_{t})$.
The kernel (\ref{eq:hat h_t^beta}) in the spectral-domain has been considered by \citet{major:1981:limit} in the special case where $\widehat{g}(\mathbf{u})=c\prod_{j=1}^k |u_j|^{-d}$ is the kernel for the spectral representation of Hermite process.
\end{Rem}

\section{Discrete chaos processes}\label{Sec:poly process}
In this section, we introduce a class of stationary sequence which converges to a  generalized Hermite process of Class (L) as defined in Definition \ref{Def:class limit}.

First we define the \emph{discrete chaos}, or the \emph{discrete multiple stochastic integral},  $Q_k(\cdot ;\mathbd{\epsilon})$ with respect to the i.i.d.\ noise $\mathbd{\epsilon}:=(\epsilon_i,i\in \mathbb{Z})$.

Let $h$ be a function defined in $\mathbb{Z}^k$ such that $\sum'_{\mathbf{i}\in \mathbb{Z}^k} h(\mathbf{i})^2<\infty$, where $'$ indicate the exclusion of the diagonals $i_p=i_q$, $p\neq q$. The following sum
\begin{align}\label{eq:Q_k(h)}
Q_k(h)=Q_k(h,\mathbd{\epsilon})=\sum'_{(i_1,\ldots,i_k)\in \mathbb{Z}^k} h(i_1,\ldots,i_k) \epsilon_{i_1}\ldots \epsilon_{i_k}=\sum'_{\mathbf{i}\in \mathbb{Z}^k} h(\mathbf{i})\prod_{p=1}^k \epsilon_{i_p},
\end{align}
is called the \emph{discrete chaos} of order $k$. It is easy to see that switching the arguments, say $i_p$ and $i_q$, $p\neq q$, of $h(i_1,\ldots,i_k)$, does not change $Q_k(h)$. So if $\tilde{h}$ is the symmetrization $h$, then $Q_k(h)=Q_k(\tilde{h})$.

The discrete chaos is related to Wiener chaos by a limit theorem.
Suppose now   we have  a sequence of function vectors $\mathbf{h}_n=(h_{1,n},\ldots,h_{j,n})$ where each $h_{j,n}\in L^2(\mathbb{Z}^{k_j})$, $j=1,\ldots,J$.   The following proposition concerns the convergence of the discrete chaos to the Wiener chaos:
\begin{Pro}\label{Pro:Poly->Wiener}
Let $\tilde{h}_{j,n}(\mathbf{x})=n^{k_j/2}h_{j,n}\left([n\mathbf{x}]+\mathbf{c}_j\right)$, $j=1,\ldots,J$, where  $\mathbf{c}_j\in \mathbb{Z}^k$. Suppose that there exists $h_j\in L^2(\mathbb{R}^{k_j})$, such that
$$
\|\tilde{h}_{j,n}-h_j\|_{L^2(\mathbb{R}^{k_j})}\rightarrow 0
$$
 as $n\rightarrow\infty$.
 Then, as $n\rightarrow\infty$,
\begin{align*}
\mathbf{Q}:=\Big(Q_{k_1}(h_{1,n}),\ldots,Q_{k_J}(h_{J,n})\Big) \ConvD \mathbf{I}:=\Big(I_{k_1}(h_1),\ldots,I_{k_J}(h_J)\Big),
\end{align*}
where each $I_{k_j}(\cdot)$, $j=1,\ldots,J$, denotes  the $k_j$-tuple Wiener-It\^o integral with respect to the same standard Brownian motion $W$.
\end{Pro}
For a proof, we refer the reader to the proof of Proposition 14.3.2 of \citet{giraitis:koul:surgailis:2009:large} on the univariate case. The proof for the multivariate case (corresponding to Proposition 14.3.3 of \citet{giraitis:koul:surgailis:2009:large}) is similar once the Cr\'amer-Wald Device is applied. The difference between Proposition \ref{Pro:Poly->Wiener} and  Proposition 14.3.3 of \citet{giraitis:koul:surgailis:2009:large} is that we  add the shift $\mathbf{c}_j$ for more flexibility. This extension requires only an easy modification to the proof.

The causal \emph{discrete chaos process}  of order $k\ge 1$ is a stationary sequence $\{X(n),n\in \mathbb{Z}\}$  defined by:
\begin{equation}\label{eq:Def PolyProcess}
X(n)=\sum'_{0<i_1,\ldots,i_k<\infty} a(i_1,\ldots,i_k)\epsilon_{n-i_1}\ldots \epsilon_{n-i_k}=\sum'_{-\infty<i_1,\ldots,i_k<n} a(n-i_1,\ldots,n-i_k)\epsilon_{i_1}\ldots \epsilon_{i_k},
\end{equation}
where $'$ indicates that the sum excludes the diagonals $i_p=i_q$, $p\neq q$, $\{\epsilon_n\}$ is an i.i.d.\ sequence with mean $0$ and variance $1$, $a(\mathbf{i})$ is a function on $\mathbb{Z}^k$, and we require that it  satisfies $\sum'_{\mathbf{i}>\mathbf{0}}  a(\mathbf{i})^2<\infty$, so that $X(n)$ is well-defined in the $L^2(\Omega)$-sense.
Note that when $k=1$, $X(n)$ is plainly a linear process.

Due to the off-diagonality, the autocovariance of $\{X(n)\}$ is given by the simple formula
\begin{equation}\label{eq:ACF}
\gamma(n):=\Cov(X(n),X(0))=k!\sum'_{\mathbf{i}>  \mathbf{0}}\tilde{a}(\mathbf{i})\tilde{a}(\mathbf{i}+|n|\mathbf{1}),
\end{equation}
 where $\tilde{a}(\cdot)$ is the symmetrization of $a(\cdot)$.

We now focus on the following case:
\begin{equation}\label{eq:a=gL}
a(\mathbf{i})=g(\mathbf{i})L(\mathbf{i}),
\end{equation}
where $g$ is a  generalized Hermite kernel of Class (L) defined in Definition \ref{Def:class limit}, and $L$ is a bounded function on $\mathbb{Z}_+^k$ which satisfies the following: for any $\mathbf{x}\in\mathbb{R}_+^k$ and for any bounded $\mathbb{Z}^k$-valued function $\mathbf{B}(\cdot)$ defined on $\mathbb{Z}_+$, we have
\begin{equation}\label{eq:L(ni)->1}
L([n\mathbf{x}]+\mathbf{B}(n))\rightarrow 1,~\text{as }n\rightarrow \infty.
\end{equation}
Note that $X(n)$ is well-defined in $L^2(\Omega)$ since $\sum_{\mathbf{i}\in \mathbb{Z}^k_+} g^*(\mathbf{i})^2<\infty$, where $g^*$ is a linear combination of terms of the form $\prod_{j=1}^k x_j^{\gamma_j}$ with every $\gamma_j<-1/2$,
\begin{Rem}
Note that the boundedness of $L$ and (\ref{eq:L(ni)->1}) are strictly weaker than assuming that $L(\mathbf{i})\rightarrow 1$ as $\|\mathbf{i}\|\rightarrow \infty$ for some norm $\|\cdot\|$ on $\mathbb{R}^k$ (recall that norms are equivalent in the finite-dimensional space). Indeed, consider
$$
L(i_1,i_2)=
\begin{cases}
2 & \text{ if } i_2=1;\\
1   & \text{ otherwise}.
\end{cases}
$$
Suppose that $\mathbf{B}$ is bounded by $M$.
Then $L([n\mathbf{x}]+\mathbf{B}(n))=1$ for large $n$. On the other hand, consider $\|\mathbf{i}\|=\max(i_1,i_2)$. Then if $(i_1,i_2)=(i_1,1)$, $i_1\rightarrow \infty$, we have $\|\mathbf{i}\|=i_1\rightarrow\infty$ but $L(i_1,i_2)=L(i_1,1)=2$.
\end{Rem}
\begin{Rem}
In practice, Relation
(\ref{eq:L(ni)->1}) implies that for any fixed $\mathbf{x}\in \mathbb{R}_+^k$ and $\mathbf{c}\in \mathbb{Z}_+^k$, $L([n\mathbf{x}]+\mathbf{c})\rightarrow 1$ as $n\rightarrow \infty$.
\end{Rem}

The following Proposition shows that one can get long-range dependence if $g$ is of Class (L).
\begin{Pro}\label{Pro:LRD ACF}
If $a(\mathbf{i})$ is as given in (\ref{eq:a=gL}), where  $g$ has homogeneity exponent $\alpha\in (-1/2-k/2,-k/2)$ (or $2\alpha+k\in (-1,0)$), then the autocovariance of the  discrete chaos process $\{X(n)\}$ satisfies
\begin{equation}\label{eq:ACF asymp}
\gamma(n)\sim k! C_{\tilde{g}} n^{2H-2}, \text{ as }n\rightarrow\infty,
\end{equation}
where $C_{\tilde{g}}=\int_{\mathbb{R}_+^k}\tilde{g}(\mathbf{x})\tilde{g}(\mathbf{1}+\mathbf{x})>0$, $H=\alpha+k/2+1\in(1/2,1)$, with $\tilde{g}$ being the symmetrization of $g$. In addition, as $N\rightarrow\infty$,
\begin{equation}\label{eq:Var asymp}
\Var[\sum_{n=1}^NX(n)] \sim \frac{k!C_{\tilde{g}}}{H(2H-1)} N^{2H}.
\end{equation}
\end{Pro}
\begin{proof}
Assume without loss of generality that $g$ is already symmetric.
\begin{align*}
(k!)^{-1}\gamma(n)
=&\sum'_{\mathbf{i}>\mathbf{0}} g(\mathbf{i})g(n\mathbf{1}+\mathbf{i})L(n\mathbf{1}+\mathbf{i})L(\mathbf{i})\\
=&n^{2\alpha+k}\sum'_{\mathbf{i}>\mathbf{0}} g\left(
\frac{\mathbf{i}}{n}\right)g\left(\mathbf{1}+\frac{\mathbf{i}}{n}\right)L(\mathbf{i})L(n\mathbf{1}+\mathbf{i})\frac{1}{n^k}\\
=&n^{2\alpha+k} \int_{\mathbb{R}^k_+}\mathrm{1}_{D_n^c}(\mathbf{x})g_n(\mathbf{x})g_n(1+\mathbf{x})d\mathbf{x},
\end{align*}
where $g_n(\mathbf{x})=g(\frac{[n\mathbf{x}]+\mathbf{1}}{n})L([n\mathbf{x}]+\mathbf{1})$, $D_n^c=\{\mathbf{x}\in\mathbb{R}_+^k, ~[nx_p]\neq [nx_q],~p\neq q\in \{1,\ldots,k\}\}$. Note that $\mathrm{1}_{D_n}(\mathbf{x})=1$ as $n$ becomes large enough, for any  $\mathbf{x}\in D^c:=\{\mathbf{x}\in \mathbb{R}_+^k, x_p\neq x_q, ~p\neq q\in \{1,\ldots,k\}\}$, and that the diagonal set $D:=\mathbb{R}_+^k\setminus D^c$ has measure $0$.
Since $g$ belongs to Class (L), $g$ is continuous a.e., so  $g_n(\mathbf{x})\rightarrow g(\mathbf{x})$ a.e.\ as $n\rightarrow\infty$. Furthermore, there exists $g^*(\mathbf{x})$ which is a linear combination of the form $\prod_{j=1}^k x_j^{\gamma_j}$ (Condition 2 of Definition \ref{Def:class limit}), so that for a.e.\ $\mathbf{x}\in \mathbb{R}_+^k$,
$$
|g_n(\mathbf{x})|\le g^*\left(\frac{[n\mathbf{x}]+\mathbf{1}}{n}\right)\le g^*(\mathbf{x}),
$$
since  $L$ is bounded and $g^*$ is decreasing in its every variable. Note that $\int_{\mathbb{R}_+^k} g^*(\mathbf{x})g^*(\mathbf{1}+\mathbf{x})d\mathbf{x}<\infty$, and $g$ is a.e.\ continuous. So it remains to apply the Dominated Convergence Theorem.

Finally, (\ref{eq:Var asymp}) follows by first noting that
$$
\Var[\sum_{n=1}^NX(n)]=\sum_{n}(N-|n|)\gamma(n)=N\sum_{|n|<N} \gamma(n)-\sum_{|n|<N} |n|\gamma(n),
$$
and then using the asymptotics of $\gamma(n)$ just derived.

\end{proof}

\section{Hypercontractivity for infinite discrete chaos}\label{Sec:hypercontract}
Let $X_M$ be a  finite discrete chaos defined as
\begin{equation}\label{eq:finite chaos}
X_M=\sum_{-M\mathbf{1}\le  \mathbf{i}\le M\mathbf{1}}' h(\mathbf{i})  \epsilon_{i_1}\ldots \epsilon_{i_k},
\end{equation}
where $h(\mathbf{i})=h(i_1,\ldots,i_k)$ is a function on $\mathbb{Z}^k$,  $M\in \mathbb{Z}_+$, and we assume that $\{\epsilon_i\}$ is a sequence of i.i.d.\ variables with $\E \epsilon_i=0$, $\E \epsilon_i^2=1$. Then we have the following moment-comparison inequality, also called ``hypercontractivity inequality'':
\begin{Pro}
Suppose that $\E |\epsilon_i|^p <\infty$ with $p\ge 2$. Then
\begin{equation}\label{eq:hypercontrac}
\E [|X_M|^p]^{1/p} \le d_{p,k} \E [|X_M|^2]^{1/2},
\end{equation}
where $d_{p,k}$ is a constant depending only on $p$ and $k$.
\end{Pro}
For a proof of (\ref{eq:hypercontrac}), where $M$ is finite, see Lemma 4.3 of \citet{krakowiak:szulga:1986:random}, where the so-called MPZ($p$) condition (Definition 1.5 of \citet{krakowiak:szulga:1986:random}) is trivially satisfied since the $\epsilon_i$'s are identically distributed.

Now we  extend  (\ref{eq:hypercontrac}) to the case $M=\infty$. The result is used in Theorem \ref{Thm:CLT}, \ref{Thm:Frac NCLT beta<0} and \ref{Thm:multi limit} below for proving tightness in $D[0,1]$.
\begin{Pro}\label{Pro:Hypercontract}
Suppose that $\sum_{\mathbf{i}\in \mathbb{Z}^k}'h(\mathbf{i})^2<\infty$.
Let $X=\sum_{\mathbf{i}\in \mathbb{Z}^k}'h(\mathbf{i}) \prod_{p=1}^k\epsilon_{i_p}$.
If for some $p'>p>2$, $\E |\epsilon_i|^{p'}<\infty$, then one has
\begin{equation}\label{eq:inf hypercontrac}
\E [|X|^p]^{1/p} \le d_{p,k} \E [|X|^2]^{1/2}
\end{equation}
\end{Pro}
\begin{proof}
Let $X_M$ be the truncated finite chaos as in (\ref{eq:finite chaos}). The condition on $h$ implies that $X_M\rightarrow X$ in $L^2(\Omega)$. Moreover,  one has by (\ref{eq:hypercontrac}),
\[
\E [|X_M|^{p'}] \le d_{p',k}^{p'} \E [|X_M|^2]^{p'/2} \le d_{p',k}^{p'}  \left(\sum_{\mathbf{i}\in \mathbb{Z}^k}' h(\mathbf{i})^2 \right)^{p'/2}.
\]
This implies that $\{|X_M|^p,M\ge 1\}$ and $\{|X_M|^2,M\ge 1\}$ are uniformly integrable, implying  convergence of the corresponding moments. So one can then let $M\rightarrow \infty$ on both sides of (\ref{eq:hypercontrac}) and obtain (\ref{eq:inf hypercontrac}).
\end{proof}

\section{Joint convergence of the discrete chaoses}
Our goal here is to obtain non-central limit theorems for the discrete chaos process introduced in Section \ref{Sec:poly process}. We shall, in fact, prove both a central limit theorem for the SRD case (getting Brownian motion as limit)  and a non-central limit theorem for the LRD case (getting the generalized Hermite process introduced in Section \ref{Sec:GenHermProc} as limit).
We also consider non-central limit theorems leading to  the fractionally filtered generalized Hermite process introduced in Section \ref{Subsec:frac}. Finally, we derive a multivariate limit theorem which mixes central and non-central limit theorems.

We  first  define here precisely what SRD and LRD stand for in the context of  discrete chaos process. Recall that $\tilde{a}(\cdot)$ denotes the symmetrization of $a(\cdot)$.
\begin{Def}\label{Def:SRD LRD}
We say a discrete chaos process $\{X(n)\}$ given in (\ref{eq:Def PolyProcess}) is
\begin{itemize}
\item SRD, if $\sum_{n=-\infty}^{\infty} \sum_{\mathbf{i}>\mathbf{0}}' |\tilde{a}(\mathbf{i})\tilde{a}(\mathbf{i}+|n|\mathbf{1})|<\infty$ and $\sum_{n=-\infty}^\infty \gamma(n)>0$;
\item LRD, if $a(\mathbf{i})=g(\mathbf{i})L(\mathbf{i})$  as given in (\ref{eq:a=gL}). In particular, $g$ is a generalized Hermite kernel of Class (L).
\end{itemize}
\begin{Rem}
The definitions of SRD and LRD in Definition \ref{Def:SRD LRD} are  distinct. Indeed, the SRD condition implies that $\sum_n|\gamma(n)|<\infty$, while LRD  yields   $\sum_n|\gamma(n)|=\infty$ by
Proposition \ref{Pro:LRD ACF}.
\end{Rem}
\end{Def}
\subsection{Central limit theorem}\label{Subsec:CLT}

\begin{Thm}\label{Thm:CLT}
If a discrete chaos process $\{X(n)\}$ given in (\ref{eq:Def PolyProcess}) is SRD in the sense of Definition \ref{Def:SRD LRD}, then
\begin{equation}\label{eq:partial sum SRD}
\frac{1}{N^{1/2}}\sum_{n=1}^{[Nt]}X(n)\ConvFDD \sigma B(t)
\end{equation}
where $B(t)$ is a standard Brownian motion, and $\sigma^2=\sum_{n=-\infty}^\infty \gamma(n)$.
\end{Thm}
\begin{proof}
Assume without loss of generality that $a(\cdot)$ is symmetric. The proof is similar to the proof of Theorem \ref{Thm:Polyform} found on p.108 of \citet{giraitis:koul:surgailis:2009:large}, so we  give only a sketch. The central idea is to introduce the $m$-truncation of $X(n)$, namely, $X^{(m)}(n):=\sum'_{\mathbf{0}<\mathbf{i}\le m\mathbf{1}}a(\mathbf{i})\prod_{j=1}^k\epsilon_{n-i_j}$, and then let $m\rightarrow\infty$. The sequence $\{X^{(m)}(n),n\in \mathbb{Z}\}$ is  $m$-dependent, so the classical invariance principle applies (\citet{billingsley:1956:invariance} Theorem 5.2). The long-run variance $\sigma^2=\sum_n\gamma(n)$ is a standard result.  We now check that the $L^2(\Omega)$ approximation is valid as $m\rightarrow\infty$, that is,
\begin{equation}\label{eq:m->inf}
\lim_{m\rightarrow\infty}\sup_{N\in \mathbb{Z}_+}\Var[Y^{(m)}_N(t)-Y_N(t)]=0, ~t>0,
\end{equation}
where $Y^{(m)}_N(t)=\frac{1}{\sqrt{N}}\sum_{n=1}^{[Nt]}X^{(m)}(n)$ and $Y_N(t)=\frac{1}{\sqrt{N}}\sum_{n=1}^{[Nt]}X(n)$, which is similar to (4.8.7) of \citet{giraitis:koul:surgailis:2009:large}. Indeed,
\begin{align}\label{eq:m->inf expr}
\Var[Y_N^{(m)}(t)-Y_{N}(t)]&=
\frac{1}{N}\Var\left[\sum_{n=1}^{[Nt]} (X^{(m)}_n-X_n)\right]=\frac{[Nt]}{N}\sum_{|n|<[Nt]}\gamma_{m}(n)(1-\frac{|n|}{[Nt]})\le t\sum_{n=-\infty}^\infty |\gamma_m(n)|,
\end{align}
where
\begin{align*}
\gamma_m(n):=\E(X_{n}-X_{n}^{(m)})(X_{0}-X_{0}^{(m)})
=k!\sum'_{\mathbf{i}>m\mathbf{1}}a(\mathbf{i})a(n\mathbf{1}+\mathbf{i}).
\end{align*}
For a fixed  $n\in \mathbb{Z}$, $\gamma_m(n)\rightarrow 0$ as $m\rightarrow\infty$, and $|\gamma_m(n)|\le \rho(n)$, where
$
\rho(n)=k!\sum'_{\mathbf{i}>\mathbf{0}}|a(\mathbf{i})a(\mathbf{i}+n\mathbf{1})|,
$
which satisfies $\sum_n \rho(n)<\infty$ by the SRD assumption in Definition \ref{Def:SRD LRD}. Since the bound in (\ref{eq:m->inf expr}) does not depend on $N$, the Dominated Convergence Theorem applies and thus (\ref{eq:m->inf}) holds.
\end{proof}

To strengthen the conclusion of Theorem \ref{Thm:CLT} to  weak convergence, we have to make some additional assumptions to prove tightness.
\begin{Thm}\label{Thm:CLTWeak}
Theorem \ref{Thm:CLT} holds with $\ConvFDD$ replaced by weak convergence $\Rightarrow$ in $D[0,1]$, if either of the following holds:
\begin{enumerate}
\item There exists $\delta>0$, such that $\E(|\epsilon_i|^{2+\delta})<\infty$;
\item There exists an $M>0$ such that $a(\mathbf{i})=0$ whenever  $\mathbf{i}>M\mathbf{1}$.
\end{enumerate}
\end{Thm}
\begin{proof}
Look first at case 1. Let
\[Y_N(t):=\frac{1}{\sqrt{N}}\sum_{n=1}^{[Nt]}X(n)\]
Select $p\in (2,2+\delta)$. By Proposition \ref{Pro:Hypercontract}, one has
\begin{equation}\label{eq:hypercontract}
\E[|Y_N(t)-Y_N(s)|^{p}]\le  c \E[|Y_N(t)-Y_N(s)|^2]^{p/2},
\end{equation}
where $c$ is some constant which doesn't depend on $s,t$ or $N$. Note that $\sum_n |\gamma(n)|<\infty$  due to SRD assumption, we have
\begin{align}\label{eq:secondmoment}
&\E\left[\left|Y_N(t)-Y_N(s)\right|^2\right]=\frac{1}{N}\E[|\sum_{n=1}^{[Nt]-[Ns]}X(n)|^2]
\notag\\=& \frac{[Nt]-[Ns]}{N}\sum_{|n|<[Nt]-[Ns]}\left(1-\frac{|n|}{[Nt]-[Ns]}\right)\gamma(n)\le \frac{[Nt]-[Ns]}{N}\sum_{n=-\infty}^\infty |\gamma(n)|.
\end{align}
Combining (\ref{eq:hypercontract}) and (\ref{eq:secondmoment}), we  have for some constant $C>0$ that
\[
\E[|Y_N(t)-Y_N(s)|^p]\le cE[|Y_N(t)-Y_N(s)|^2]^{p/2}\le C |F_N(t)-F_N(s)|^{p/2},
\]
where $F_N(t)=[Nt]/N$. Now by applying Lemma 4.4.1 and Theorem 4.4.1 of \citet{giraitis:koul:surgailis:2009:large}, noting that $p/2>1$, we conclude that tightness holds.

For case 2, $X(n)$ is $M$-dependent, so by Theorem 5.2 of \citet{billingsley:1956:invariance} tightness holds as well.

\end{proof}

\subsection{Non-central limit theorem}\label{Subsec:NCLT}
The following theorem shows that in the LRD case, the discrete chaos process converges weakly to a generalized Hermite process.

\begin{Thm}\label{Thm:NCLT}
If a discrete chaos process $\{X(n)\}$ given in (\ref{eq:Def PolyProcess}) is LRD in the sense of Definition \ref{Def:SRD LRD}, then
\begin{equation}\label{eq:partial sum LRD}
\frac{1}{N^H}\sum_{n=1}^{[Nt]}X(n)\Rightarrow Z(t),
\end{equation}
in $D[0,1]$, where $Z(t)$ is the generalized Hermite process in (\ref{eq:Z(t)}), and $$H=\alpha+k/2+1\in \left(\frac{1}{2},1\right),$$
where $\alpha\in (-1/2-k/2,-k/2)$ is the homogeneity exponent of $g$ and $k$ is the order of $\{X(n)\}$.
\end{Thm}
\begin{proof}
Tightness in $D[0,1]$ is standard since $H>1/2$. We only need to show convergence in finite-dimensional distributions. Assume for simplicity that $a(\mathbf{i})=g(\mathbf{i})$ or equivalently $L(\mathbf{i})=1$. The inclusion of a general $L$ can be done  as in the proof of Proposition \ref{Pro:LRD ACF}.  We want to show that
\begin{align}\label{eq:Q_k(h)->I}
\frac{1}{N^H}\sum_{n=1}^{[Nt]} X(n)=\sum'_{(i_1,\ldots,i_k)\in \mathbb{Z}^k}\frac{1}{N^{\alpha+k/2+1}}\sum_{n=1}^{[Nt]}g(n\mathbf{1}-\mathbf{i})\mathrm{1}_{\{n\mathbf{1}>\mathbf{i}\}}\epsilon_{i_1}\ldots\epsilon_{i_k}=:Q_k(h_{t,N})\ConvFDD Z(t),
\end{align}
where $Q_k(\cdot)$ is defined in (\ref{eq:Q_k(h)}). Now in view of Proposition \ref{Pro:Poly->Wiener}, we only need to check that
\begin{equation}\label{eq:h tilde L2 conv h}
\|\tilde{h}_{t,N}(\mathbf{x})-h_{t}(\mathbf{x})\|_{L^2(\mathbb{R}^k)}\rightarrow 0,
\end{equation}
where 
\[h_t(\mathbf{x})=\int_0^t g(s\mathbf{1}-\mathbf{x})\mathrm{1}_{\{s\mathbf{1}>\mathbf{x}\}}ds,\]
 and
\begin{align*}
\tilde{h}_{t,N}(\mathbf{x}):&=N^{k/2} h_{t,N}([N\mathbf{x}]+\mathbf{1})
=\frac{1}{N^{\alpha+1}}\sum_{n=1}^{[Nt]}g(n\mathbf{1}-[N\mathbf{x}]-\mathbf{1})\mathrm{1}_{\{n\mathbf{1}>[N\mathbf{x}]+\mathbf{1}\}}\\&=
\sum_{n=1}^{[Nt]}g\left(\frac{n\mathbf{1}-[N\mathbf{x}]-\mathbf{1}}{N}\right)\mathrm{1}_{\{n\mathbf{1}>[N\mathbf{x}]+\mathbf{1}\}} \frac{1}{N}=\int_0^t g\left(\frac{[Ns\mathbf{1}]-[N\mathbf{x}]}{N}\right)\mathrm{1}_{\{[Ns\mathbf{1}]> [N\mathbf{x}]\}} ds -  R_N(t,\mathbf{x}).
\end{align*}
where 
\[R_N(t,\mathbf{x})=\frac{Nt-[Nt]}{N}g\left(\frac{[Nt\mathbf{1}]-[N\mathbf{x}]}{N}\right)\mathrm{1}_{\{[Nt\mathbf{1}]> [N\mathbf{x}]\}}.\]
 Note that
we have replaced $\mathbf{i}$ by $[N\mathbf{x}]+\mathbf{1}$ and $n$ by $[Ns]+1$. By Condition 2 in Definition \ref{Def:class limit}, there exists a positive generalized Hermite kernel $g^*(\mathbf{x})$ which is a linear combination of the form $\prod_{j=1}^k x_j^{\gamma_j}$, such that $|g(\mathbf{x})|\le g^*(\mathbf{x})$ for a.e.\ $\mathbf{x}\in \mathbb{R}_+^k$.  We assume without loss of generality that $g^*(\mathbf{x})=\prod_{j=1}^k x_j^{\gamma_j}$. Since $[Ns\mathbf{1}]>[N\mathbf{x}]$ implies $s\mathbf{1}>\mathbf{x}$, we have
\begin{equation}\label{eq:g<=g^*}
\left|g\left(\frac{[Ns\mathbf{1}]-[N\mathbf{x}]}{N}\right)\right|\mathrm{1}_{\{[Ns\mathbf{1}]>[N\mathbf{x}]\}}\le
\left(\prod_{j=1}^k\left(\frac{[Ns]-[Nx_j]}{N}\right)^{\gamma_j}\mathrm{1}_{\{[Ns]>[Nx_j]\}}\right)\mathrm{1}_{\{s\mathbf{1}>\mathbf{x}\}} ~a.e..
\end{equation}
Moreover, if $0<[Ns]-[Nx]=  k\in \mathbb{Z}_+$, then $Ns-1-Nx\le k$, and hence $s-x\le \frac{k+1}{N}$. So we have for any $\gamma<0$ that
\begin{align}\label{eq:useful}
\sup_{N\ge 1,[Ns]>[Nx]} \left(\frac{[Ns]-[Nx]}{N}\right)^{\gamma}(s-x)^{-\gamma}\le &
\sup_{N\ge 1,[Ns]-[Nx]=k\ge 1} \left(\frac{k}{N}\right)^{\gamma}(s-x)^{-\gamma}\notag\\
\le& \sup_{N\ge 1,k\ge 1} \left(\frac{k}{N}\right)^{\gamma}\left(\frac{k+1}{N}\right)^{-\gamma} =2^{-\gamma}.
\end{align}
 So we have for some constant $C>0$,
\begin{equation}\label{eq:dominate}
\left|g\left(\frac{[Ns\mathbf{1}]-[N\mathbf{x}]}{N}\right)\right|\mathrm{1}_{\{[Ns\mathbf{1}]>[N\mathbf{x}]\}}\le C g^*(s\mathbf{1}-\mathbf{x})\mathrm{1}_{\{s\mathbf{1}>\mathbf{x}\}}.
\end{equation}
Since $g(\mathbf{x})$ by assumption of Class (L) is continuous a.e.,
$g\left(\frac{[Ns\mathbf{1}]-[N\mathbf{x}]}{N}\right)\mathrm{1}_{\{[Ns\mathbf{1}]>[N\mathbf{x}]\}}$ converges a.e.\ to $g(s\mathbf{1}-\mathbf{x})\mathrm{1}_{\{s\mathbf{1}>\mathbf{x}\}}$ as $N\rightarrow\infty$. In view of (\ref{eq:dominate}), and noting that  $\int_{\mathbb{R}^k}d\mathbf{x}\left(\int_0^t g^*(s\mathbf{1}-\mathbf{x})\mathrm{1}_{\{s\mathbf{1}>\mathbf{x}\}}ds\right)^2<\infty$ because $g^*$ is a generalized Hermite kernel, one then  applies the Dominated Convergence Theorem to conclude the $L^2$ convergence of $\int_0^t g\left(\frac{[Ns\mathbf{1}]-[N\mathbf{x}]}{N}\right)\mathrm{1}_{\{[Ns\mathbf{1}]> [N\mathbf{x}]\}} ds$ to $h_t(\mathbf{x})$. For the remainder term $R_{N,t}(\mathbf{x})$, 
one has
\[
\|R_{N,t}(\mathbf{x})\|_{L^2(\mathbb{R}^k)}^2= N^{-2H}(Nt-[Nt])^2 \sum_{\mathbf{i}>\mathbf{0}}  g\left(\mathbf{i}\right)^2
\rightarrow 0\]
as $N\rightarrow\infty$.
The proof is thus complete. 
\end{proof}

\begin{Eg}
Consider the kernel $g(\mathbf{x})$ defined in (\ref{eq:eg}).
It belongs to Class (L) by Example \ref{eg:L}. Hence by Theorem \ref{Thm:NCLT}, we have the following weak convergence in $D[0,1]$:
\begin{align*}
&\frac{1}{N^{H}}\sum_{n=1}^{[Nt]} \sum'_{(i_1,\ldots,i_k)\in \mathbb{Z}_+^k}   \left( \frac{ \prod_{j=1}^k i_j}{\sum_{j=1}^k i_j^{k-\alpha}}\vee \prod_{j=1}^k i_j^{\alpha/k}\right)  ~\epsilon_{n-i_1}\ldots \epsilon_{n-i_k}
\Rightarrow \\&\int_{\mathbb{R}^k}' ~  \int_0^{t} ~ \left( \frac{ \prod_{j=1}^k (s-x_j)_+}{\sum_{j=1}^k (s-x_j)_+^{k-\alpha}}\right)\vee \left(\prod_{j=1}^k (s-x_j)_+^{\alpha/k}\right) ~ds~   W(dx_1)\ldots W(dx_k),
\end{align*}
where $H=\alpha+k/2+1$.
\end{Eg}

\subsection{Non-central limit theorem with fractional filter }
In the spirit of \citet{rosenblatt:1979:some} and \citet{major:1981:limit}, we consider here the non-central limit theorem for the fractionally filtered generalized Hermite process introduced in Section \ref{Subsec:frac}.
Assume throughout that the generalized Hermite kernel $g$ is of Class (L) (Definition \ref{Def:class limit}).

\begin{Def} \label{Def:fLRD}
Let
$X(n)=\sum_{\mathbf{i}<n\mathbf{1}}' a(n\mathbf{1}-\mathbf{i})\prod_{j=1}^k \epsilon_{i_j}$ be the same discrete chaos process as in Theorem \ref{Thm:NCLT}. We say that a discrete process $U(n)$ is fLRD
(fractionally-filtered LRD discrete chaos process) if
\begin{equation}\label{eq:frac U}
U(n)=\sum_{m=1}^\infty C_m X(n-m)=\sum_{m=-\infty}^{n-1}C_{n-m}\sum_{\mathbf{i}<m\mathbf{1}}' a(m\mathbf{1}-\mathbf{i})\prod_{j=1}^k \epsilon_{i_j},
\end{equation}
where $a(\mathbf{i})=g(\mathbf{i})L(\mathbf{i})$ as in (\ref{eq:a=gL}) with $g$ being a generalized Hermite kernel in Class (L),
$$
C_n\sim c n^{\beta-1}
$$
as $ n\rightarrow\infty,$ and where, as in Proposition \ref{Pro:beta range},
\begin{equation}\label{e:beta3}
\beta\in \left(-\frac{2\alpha+k+2}{2},-\frac{2\alpha+k}{2}\right).
 \end{equation}
 \end{Def}

 $U(n)$ is well-defined in the $L^2(\Omega)$ sense. Indeed, we have the following:
\begin{Lem}\label{Lem:U well defined} We have
\begin{equation*}
\sum_{\mathbf{i}\in\mathbb{Z}^k}' \left(\sum_{m<n} |C_{n-m}a(m\mathbf{1}-\mathbf{i})|\mathrm{1}_{\{m\mathbf{1}>\mathbf{i}\}}\right)^2<\infty.
\end{equation*}
\end{Lem}
\begin{proof}
Note that $a(\cdot)=g(\cdot)L(\cdot)$, where $g$ is of Class (L). So by Definition \ref{Def:class limit}, there exists a
 $g^*(\mathbf{x})>0$  which is a finite linear combination of the form $\prod_{j=1}^k x_j^{\gamma_j}$, such that $|g(\mathbf{x})|<g^*(\mathbf{x})$. Note that $L$ is bounded and $|C_n|\le cn^{\beta-1}$. Set $n=-1$ without loss of generality due to stationarity. We hence need to show that
\begin{equation}\label{eq:square sum frac}
\sum_{\mathbf{i}\in\mathbb{Z}^k} \left(\sum_{m<-1} (-m)^{\beta-1} g^*(m\mathbf{1}-\mathbf{i})1_{\{m\mathbf{1}>\mathbf{i}\}}\right)^2<\infty.
\end{equation}
It suffices to show this when $\beta>0$, since for any $\beta'\le 0$ and $\beta>0$,  $(-m)^{\beta'-1}\le  (-m)^{\beta-1}$ for all $m<-1$.
The preceding sum can be rewritten as an integral by replacing $m$ by $[s]$ and $\mathbf{i}$ by $[\mathbf{x}]$:
\begin{align}\label{eq:int approx sum}
\int_{\mathbb{R}^k}1_{D^c} d\mathbf{x}\left( \int_{-\infty}^{-1} ds (-[s])^{\beta-1} g^*([s\mathbf{1}]-[\mathbf{x}])1_{\{[s\mathbf{1}]>[\mathbf{x}]\}} \right)^2,
\end{align}
where $D^c=\{\mathbf{x}\in \mathbb{R}^k:~ [x_p]\neq [x_q], ~p\neq q\}$. By $[s]\le s$, $\beta-1<0$, and
(\ref{eq:dominate}),  (\ref{eq:int approx sum}) is bounded by (up to a constant)
\begin{align*}
&\int_{\mathbb{R}^k} d\mathbf{x}\left( \int_{-\infty}^{-1} ds (-s)_+^{\beta-1} g^*(s\mathbf{1}-\mathbf{x})1_{\{s\mathbf{1}>\mathbf{x}\}} \right)^2\\
=&\int_{-\infty}^{-1} ds(-s)^{\beta-1}\int_{0}^{-s}du(-s-u)^{\beta-1}u^{2\alpha+k} \int_{\mathbb{R}_+^k}d\mathbf{y}g^*(\mathbf{y})g^*(\mathbf{1}+\mathbf{y})\\
=& \int_{1}^{\infty}s^{2\alpha+2\beta+k-1}ds ~\mathrm{B}(\beta,2\alpha+k+1) ~ C_{g^*}<\infty,
\end{align*}
where we have used a change of variable similar to the lines below (\ref{eq:int h^beta_t}), and in addition the assumptions $\beta>0$, $2\alpha+k>-1$,  $2\alpha+2\beta+k<0$, and $g^*$ is a  generalized Hermite kernel.

\end{proof}
\begin{Rem}
Lemma \ref{Lem:U well defined} not only shows that $U(n)$ is well-defined in $L^2(\Omega)$, it also allows changing the order of summations, which will be used in proving the non-central limit theorem below.
\end{Rem}

Next we want to obtain  non-central limit theorems, that is, to show that the suitably normalized partial sum of $U(n)$ defined in (\ref{eq:frac U}) converges to the fractionally-filtered generalized Hermite process introduced in Section \ref{Subsec:frac}.  We need to distinguish two cases: $\beta>0$ (which increases $H$) and $\beta<0$ (which decreases $H$).

We first consider $\beta>0$:
\begin{Thm}\label{Thm:frac non central beta>0}
Let $U(n)$ be as in (\ref{eq:frac U}) with $\beta\in (0,-\alpha-k/2)$. Then
\begin{align*}
\frac{1}{N^{H}}\sum_{n=1}^{[Nt]} U(n) \Rightarrow Z^\beta(t),
\end{align*}
where
$$1/2<\alpha+k/2+1<H=\alpha+\beta+k/2+1<1,$$
 and $Z^\beta(t)$ is the fractionally-filtered generalized Hermite process defined in Theorem \ref{Thm:Frac Z beta hsssi}. It is defined using the same $g$ and $\beta$ as $U(n)$.
\end{Thm}
\begin{proof} Since $H>1/2$, tightness in $D[0,1]$ is standard. We now show convergence in finite-dimensional distributions.
Assume for simplicity that $C_m=m^{\beta-1}$ and $L(\mathbf{i})=1$.  By Lemma \ref{Lem:U well defined}, we are able to change the order of the summations to write:
\begin{align*}
\frac{1}{N^H}\sum_{n=1}^{[Nt]}U(n)=\sum_{\mathbf{i}\in \mathbb{Z}^k}' \frac{1}{N^H}\sum_{n=1}^{[Nt]}\sum_{m<n} (n-m)^{\beta-1} g(m\mathbf{1}-\mathbf{i})\mathrm{1}_{\{m\mathbf{1}>\mathbf{i}\}}\prod_{j=1}^k \epsilon_{i_j}=\sum_{\mathbf{i}\in \mathbb{Z}^k} h_{t,N}^\beta(\mathbf{i})\prod_{j=1}^k \epsilon_{i_j} =Q_k(h^\beta_{t,N}),
\end{align*}
and by setting $\tilde{h}^\beta_{t,N}(\mathbf{x})=N^{k/2}h^\beta_{t,N}([N\mathbf{x}]+\mathbf{1})$, we have
\begin{align*}
\tilde{h}^\beta_{t,N}(\mathbf{x})=&\frac{1}{N^{\alpha+\beta+1}}\sum_{n=1}^{[Nt]} \sum_{m<n}(n-m)^{\beta-1}g\left(m\mathbf{1}-[N\mathbf{x}]-\mathbf{1}\right)\mathrm{1}_{\{m\mathbf{1}>[N\mathbf{x}]-\mathbf{1}\}}\\=&
\sum_{n=1}^{[Nt]} \sum_{m<n} \left(\frac{n-m}{N}\right)^{\beta-1} g\left(\frac{m\mathbf{1}-[N\mathbf{x}]-\mathbf{1}}{N}\right)\mathrm{1}_{\{m\mathbf{1}>[N\mathbf{x}]-\mathbf{1}\}} \frac{1}{N^2}
\\=&\int_0^t ds \int_\mathbb{R} dr  \left(\frac{[Ns]-[Nr]}{N}\right)_+^{\beta-1}g\left(\frac{[Nr\mathbf{1}]-[N\mathbf{x}]}{N}\right)\mathrm{1}_{\{[Nr\mathbf{1}]> [N\mathbf{x}]\}}-R_{N,t}(\mathbf{x})
\\=&:\int_0^t ds \int_\mathbb{R} dr G_{N}(s,r,\mathbf{x})\mathrm{1}_{K_N}-R_{N,t}(\mathbf{x})
\end{align*}
where we associate $\mathbf{i}$ with $[N\mathbf{x}]+\mathbf{1}$, $n$  with $[Ns]+1$, and $m$  with $[Nr]+1$,
$$
G_N(s,r,\mathbf{x}):=
\left(\frac{[Ns]-[Nr]}{N}\right)^{\beta-1}g\left(\frac{[Nr\mathbf{1}]-[N\mathbf{x}]}{N}\right),
$$
 $$
 K_N=\{[Ns]> [Nr],[Nr\mathbf{1}]> [N\mathbf{x}]\}\subset\{s>r,r\mathbf{1}>\mathbf{x}\},
  $$
and
\[
R_{N,t}(\mathbf{x})=\frac{Nt-[Nt]}{N}\int_\mathbb{R} dr  \left(\frac{[Nt]-[Nr]}{N}\right)_+^{\beta-1}g\left(\frac{[Nr\mathbf{1}]-[N\mathbf{x}]}{N}\right)\mathrm{1}_{\{[Nr\mathbf{1}]> [N\mathbf{x}]\}}.
\]
  In view of Proposition \ref{Pro:Poly->Wiener}, we need to show that $\tilde{h}^\beta_{t,N}\rightarrow h^\beta_t$ and $R_{N,t}\rightarrow 0$ in $L^2(\mathbb{R}^k)$, where
  $$
  h_t^\beta(\mathbf{x}):=\int_0^tds\int_\mathbb{R} dr (s-r)_+^{\beta-1}g(r\mathbf{1}-\mathbf{x})\mathrm{1}_{\{r\mathbf{1}>\mathbf{x}\}}.
  $$
Using (\ref{eq:g<=g^*}) and (\ref{eq:useful}) (note that $\beta-1<0$) as in the proof of Theorem \ref{Thm:NCLT}, we can bound the integrand as
\[
|G_{N}(s,r,\mathbf{x})|\mathrm{1}_{K_N}\le C (s-r)_+^{\beta-1}g^*(r\mathbf{1}-\mathbf{x})\mathrm{1}_{\{r\mathbf{1}>\mathbf{x}\}}
\]
for some $C>0$, where $g^*(\mathbf{x})$ is a generalized Hermite kernel from Definition \ref{Def:class limit}. Because
$$
h^*(\mathbf{x}):=(s-r)_+^{\beta-1}g^*(r\mathbf{1}-\mathbf{x})\mathrm{1}_{\{r\mathbf{1}>\mathbf{x}\}} \in L^2(\mathbb{R}^k)
 $$
 by (\ref{eq:h_t beta>0}) and Proposition \ref{Pro:beta range}, and $g$ is a.e.\ continuous, it remains to apply the Dominated Convergence Theorem to conclude $\tilde{h}^\beta_{t,N}\rightarrow h^\beta_t$. For the remainder term $R_{N,t}(\mathbf{x})$, one has
\begin{align*}
\|R_{N,t}(\mathbf{x})\|_{L^2(\mathbb{R}^k)}^2=N^{-2H} (Nt-[Nt])   \sum_{\mathbf{i}\in \mathbb{Z}^k}\left(\sum_{m<[Nt]}([Nt]-m)^{\beta-1}g(m\mathbf{1}-\mathbf{i})1_{\{m\mathbf{1}>\mathbf{i}\}}\right)^{2},
\end{align*}
which, in view of (\ref{eq:square sum frac}), converges to $0$ as $N\rightarrow\infty$. The proof is thus complete.
\end{proof}

We now treat the case $\beta<0$. This case is more delicate than the case $\beta>0$ in two ways: a) an additional assumption on the linear-filter response $\{C_n\}$ has to be made; b) if $\beta$ is chosen such that $H<1/2$, then tightness of the normalized partial sum process needs also additional assumptions.

When $\beta<0$, we have
$$
\sum_{n=1}^\infty |C_n|<\infty.
 $$
 If $f_X$ is the spectral density of $\{X(n)\}$, then the spectral density of $\{U(n)\}$ is
 $$
 f_U(\lambda)=|C(e^{i\lambda})|^2 f_X(\lambda),
  $$
  where $C(z):=\sum_n C_nz^n$,  and the transfer function $H(\lambda):=|C(e^{i\lambda})|^2$ is continuous. Since $X(n)$ is LRD (see Proposition \ref{Pro:LRD ACF}), its spectral density blows up at the origin. To dampen it we need to multiply it by an $H(\lambda)$ which converges to $0$ as $\lambda\rightarrow 0$. This means that $H(0)=|\sum_{n=1}^\infty C_n|^2=0$, and hence we need to assume $\sum_{n=1}^\infty C_n=0$.

\begin{Thm}\label{Thm:Frac NCLT beta<0}
Let $U(n)$ be as in (\ref{eq:frac U}) with $\beta\in (-\alpha-k/2-1,0)$, and assume in addition that
\begin{equation}\label{eq:sum C_n=0}
\sum_{n=1}^\infty C_n=0.
\end{equation}
Then
\begin{align*}
\frac{1}{N^{H}}\sum_{n=1}^{[Nt]} U(n) \ConvFDD Z^\beta(t),
\end{align*}
where
\[
0<H=\alpha+\beta+k/2+1<\alpha+k/2+1<1,
\]
$Z^\beta(t)$ is the fractionally-filtered generalized Hermite process defined in Theorem \ref{Thm:Frac Z beta hsssi}. It  is defined using the same $g$ and $\beta$ as $U(n)$.

If in addition, either a) $H>1/2$, or b) $H<1/2$ and for some $p>1/H$, $\mathbb{E}|\epsilon_i|^p<\infty$, then the above $\ConvFDD$ can be replaced with weak convergence  in $D[0,1]$.
\end{Thm}
\begin{proof}
Note that by Lemma \ref{Lem:U well defined},  we can change the order of summations to write:
\begin{align*}
Y_N(t):&=\frac{1}{N^H}\sum_{n=1}^{[Nt]}U(n)=
\sum_{\mathbf{i}\in \mathbb{Z}^k}'\frac{1}{N^{H}}\sum_{n=1}^{[Nt]}  \sum_{m<n}C_{n-m}\sum_{\mathbf{i}<m\mathbf{1}}' a(m\mathbf{1}-\mathbf{i})\prod_{j=1}^k \epsilon_{i_j}
\\&=\sum_{\mathbf{i}\in \mathbb{Z}^k}' \frac{1}{N^H}\sum_{m\in \mathbb{Z}}a(m\mathbf{1}-\mathbf{i})\mathrm{1}_{\{m\mathbf{1}>\mathbf{i}\}} \sum_{n=1\vee (m+1)}^{[Nt]}C_{n-m} \prod_{j=1}^k \epsilon_{i_j}=Q_k(h^\beta_{t,N}),
\end{align*}
where
\[
h_{t,N}^\beta(\mathbf{i})=\frac{1}{N^H}\sum_{m\in \mathbb{Z}}a(m\mathbf{1}-\mathbf{i})\mathrm{1}_{\{m\mathbf{1}>\mathbf{i}\}} \sum_{n=1\vee (m+1)}^{[Nt]}C_{n-m}.
\]
Making use of (\ref{eq:sum C_n=0}), and using  $l$ to denote a generic  function such that  $l(i)\rightarrow 1$ as $i\rightarrow\infty$, we have
if $m\ge 1$,
\[
\sum_{n=1\vee (m+1)}^{[Nt]}C_{n-m}=\sum_{n=1}^{[Nt]-m}C_n=-\sum_{n=[Nt]-m+1}^\infty C_n=  \beta^{-1}l([Nt]-m+1)([Nt]-m+1)_+^{\beta};
\]
and if $m\le 0$,
\begin{align*}
\sum_{n=1\vee (m+1)}^{[Nt]}C_{n-m}= &\sum_{n=1}^{[Nt]}C_{n-m}=\sum_{n=-m+1}^{[Nt]-m}C_n= \sum_{n=[Nt]-m+1}^\infty C_n
-\sum_{n=-m+1}^{\infty} C_n\\= &\beta^{-1}\left[l([Nt]-m+1)([Nt]-m+1)_+^\beta-l(-m)(-m)_+^\beta \right].
\end{align*}
So  by letting $\mathbf{i}$ correspond to $[N\mathbf{x}]+\mathbf{1}$ and $m$ to $[Ns]+1$ (omitting $L$ and $l$ for simplicity),
\begin{align*}
\tilde{h}^\beta_{t,N}(\mathbf{x})&=N^{k/2} h^\beta_{t,N}([N\mathbf{x}]+\mathbf{1})
\\&=\beta^{-1}\int_\mathbb{R}g\left(\frac{[Ns]\mathbf{1}-[N\mathbf{x}]}{N}\right)\mathrm{1}_{\{[Ns]\mathbf{1}>[N\mathbf{x}]\}}\left(\left(\frac{[Nt]-[Ns]}{N}\right)^\beta_+ - \left(\frac{-[Ns]-1}{{N}}\right)^\beta_+\right) ds.
\end{align*}
Using similar arguments as in the proof of Theorem \ref{Thm:NCLT}, we can bound the absolute value of the integrand above by $Cg^*(s\mathbf{1}-\mathbf{x})\mathrm{1}_{\{s\mathbf{1}>\mathbf{x}\}}\left((t-s)_+^\beta- (-s)_+^\beta\right)$ for some $C>0$, where $g^*$ is a generalized Hermite kernel from Definition \ref{Def:class limit} (for the last term, we use $[Ns]+1\ge Ns$). Note that $\beta<0$ in this case. By applying the Dominated Convergence Theorem,  we get the desired f.d.d.\ convergence using Proposition \ref{Pro:Poly->Wiener}.

Now we turn to the weak convergence. When
$H>1/2$, the tightness is standard.  To show tightness under condition $H<1/2$ and $\E |\epsilon_i|^p<\infty$,   Proposition \ref{Pro:Hypercontract} and  the above  f.d.d.\ convergence imply that for some  constant $c,C>0$ free from $s,t$ and $N$,
\[
\E |Y_N(t)-Y_N(s)|^{p'} \le c\E [|Y_N(t)-Y_N(s)|^2]^{p'/2} \le C |F_N(t)-F_N(s)|^{p'H},
\]
where $F_N(t)=[Nt]/N$, $p'<p$ and $p'H>1$.  Now by  Lemma 4.4.1 and Theorem 4.4.1 of \citet{giraitis:koul:surgailis:2009:large}, we conclude that tightness holds.
\end{proof}

\subsection{Mixed multivariate limit theorem}
In \citet{bai:taqqu:2013:multivariate}, a multivariate version of Theorem \ref{Thm:Polyform} is obtained, where both central and non-central convergence  appear simultaneously. We will state here a similar theorem.

Suppose that $\mathbf{X}(n)=\left(X_{1}(n),\ldots,X_{J}(n)\right)$ is a vector of discrete chaos process defined on the same noise but with different coefficients, that is,
\begin{equation}\label{eq:X_j(n)}
X_{j}(n)=\sum'_{0<i_1,\ldots,i_{k_j}<\infty} a_j(i_1,\ldots,i_{k_j}) \epsilon_{n-i_1}\ldots\epsilon_{n-i_{k_j}}=\sum'_{\mathbf{i}>\mathbf{0}} a_j(\mathbf{i}) \prod_{p=1}^{k_j}\epsilon_{n-i_p},
\end{equation}
where we assume  $\{\epsilon_i\}$  is an i.i.d.\ random sequence  with mean $0$ and variance $1$. For convenience we let  $a_j(i_1,\ldots,i_{k_j})=a_j(\mathbf{i})=a_j(\mathbf{i})\mathrm{1}_{\{\mathbf{i}>\mathbf{0}\}}$, and $\tilde{a}_j(\cdot)$ denotes the symmetrization of $a_j(\cdot)$.

\begin{Def}\label{Def:SRD LRD Multi}
We say that the vector sequence of discrete chaos processes $\{\mathbf{X}(n)\}$ is
\begin{itemize}
\item SRD, if every component $X_j(n)$ is SRD in the sense of  Definition \ref{Def:SRD LRD}, and in addition, for any $p\neq q\in \{1,\ldots,J\}$,
\begin{equation}\label{eq:cross cov bound}
\sum_{n=-\infty}^{\infty} \sum_{\mathbf{i}>\mathbf{0}}' |\tilde{a}_p(\mathbf{i})\tilde{a}_q(n\mathbf{1}+\mathbf{i})|<\infty;
\end{equation}
\item LRD, if every component $X_j(n)$ is  LRD in the sense of Definition \ref{Def:SRD LRD}.
\item fLRD, if every component $X_j(n)$ is a fractionally-filtered LRD discrete chaos process
in the sense of Definition \ref{Def:fLRD}. Note: these components were denoted $U(n)$ in that definition.
\end{itemize}
\end{Def}
\begin{Rem}
 If the vector sequence is SRD,  then (\ref{eq:cross cov bound}) guarantees that the cross-covariance $\gamma_{p,q}(n):=\Cov(X_p(n),X_q(0))$ satisfies $\sum_n |\gamma_{p,q}(n)|<\infty$. As in Proposition 2.5 of \cite{bai:taqqu:2013:multivariate}, we have that as $N\rightarrow\infty$,
\begin{equation}\label{eq:cross covariance limit}
\Cov\left(\frac{1}{\sqrt{N}}\sum_{n=1}^{[Nt_1]}X_p(n),\frac{1}{\sqrt{N}}\sum_{n=1}^{[Nt_2]} X_q(n)\right)\rightarrow (t_1\wedge t_2) \sum_{n=-\infty}^\infty \gamma_{p,q}(n).
\end{equation}
Note that $\gamma_{p,q}(n)=0$ always if the orders $k_p\neq k_q$.
\end{Rem}

We will now consider a general case where SRD and LRD and fLRD vectors can all be present in  $\mathbf{X}(n)$.
We  divide $\mathbf{X}(n)$ into four parts
$$\mathbf{X}(n)=(\mathbf{X}_{S_1}(n),\mathbf{X}_{S_2}(n),\mathbf{X}_L(n),\mathbf{X}_F(n))$$
of dimension $J_{S_1}, J_{S_2}, J_{L}, J_F$ respectively, which are defined as follows:

 \begin{enumerate}[(i)]
 \item
 all the components of $\mathbf{X}_{S_1}(n)=(X_{1,S_1}(n),\ldots,X_{J_{S_1},S_1}(n))$ have order $k=1$, namely, are all linear processes;
 \item
 every component of $\mathbf{X}_{S_2}(n)=(X_{1,S_2}(n),\ldots,X_{J_{S_2},S_2}(n))$ has order $k\ge 2$, and
the {\it combined vector}
 $$
\mathbf{X}_S(n)=(\mathbf{X}_{S_1}(n),\mathbf{X}_{S_2}(n))=(X_{1,S}(n),\ldots,X_{J_S,S}(n)),\quad J_S=J_{S_1}+J_{S_2},
$$
is SRD in the sense of Definition \ref{Def:SRD LRD Multi};
\item
the vector $\mathbf{X}_L(n)=(X_{1,L}(n),\ldots,X_{J_L,L}(n))$ is LRD  in the sense of Definition \ref{Def:SRD LRD Multi},  with correspondingly generalized Hermite kernels $\mathbf{g}=(g_{1,L},\ldots,g_{J_L,L})$;
\item
the vector $\mathbf{X}_F(n)=(X_{1,F}(n),\ldots,X_{J_F,F}(n))$ is
 fLRD in the sense of Definition \ref{Def:SRD LRD Multi}, with
 correspondingly generalized Hermite kernels $\mathbf{g}=(g_{1,F},\ldots,g_{J_F,F})$
 and fractional exponent $\mathbd{\beta}=(\beta_{1},\ldots,\beta_{J_F})$.
\end{enumerate}

We now state the multivariate limit theorem. We use $Y_N$ (with subscript $S_1$, $S_2$, $L$ or $F$) to denote the corresponding normalized sum $Y_N(t):=N^{-H}\sum_{n=1}^{[Nt]}X(n)$, where $X(n)$ is a component of $\mathbf{X}(n)$, $H$ is such that $\Var (Y_N(1))$ converges to some constant $c>0$ as $N\rightarrow\infty$.
\begin{Thm}\label{Thm:multi limit} Following the notation defined above, one has
\begin{align}\label{eq:multi conv}
(\mathbf{Y}_{N,S_1}(t),\mathbf{Y}_{N,S_2}(t),\mathbf{Y}_{N,L}(t),\mathbf{Y}_{N,F}(t))\ConvFDD(\mathbf{B}_1(t),\mathbf{B}_2(t),\mathbf{Z}(t),\mathbf{Z}^{\mathbd{\beta}}(t)),
\end{align}
where
\begin{enumerate}[(i)]
\item  $\mathbf{B}_1(t)=\mathbf{W}(t):=(\sigma_{1}W(t),\ldots,\sigma_{J_{S_1}}W(t))$ defined by the same standard Brownian motion $W(t)$, and
$$
\sigma_{p}=\sum_{n=-\infty}^\infty \sum_{i>0} a_{p,S_1}(n)a_{p,S_1}(n+i), \quad p=1,\ldots,J_{S_1}.
$$
\item $\mathbf{B}_2(t)$ is a multivariate Brownian motion with the covariance given by (\ref{eq:cross covariance limit});
\item $\mathbf{Z}(t)$ is a multivariate generalized Hermite process  defined as in (\ref{eq:Z(t)}) by the kernels $(g_{1,L},\ldots,g_{J_L,L})$ and using the $W(t)$ in Point $(i)$ as Brownian motion integrator.
\item $\mathbf{Z}^{\mathbd{\beta}}(t)$ is a multivariate fractionally-filtered generalized Hermite process  defined as in (\ref{eq:frac filt proc full}) by the kernels $(g_{1,F},\ldots,g_{J_F,F})$, fractional exponent $\mathbd{\beta}=(\beta_{1},\ldots,\beta_{J_F})$ and using the $W(t)$ in Point $(i)$ as  Brownian motion integrator.
\end{enumerate}
Moreover,  $\mathbf{B}_2(t)$  is always independent of $(\mathbf{B}_1(t),\mathbf{Z}(t),\mathbf{Z}^{\mathbd{\beta}}(t))$.

In addition, $\ConvFDD$ in (\ref{eq:multi conv}) can be replaced with weak convergence in $D[0,1]^J$, if every component of $\mathbf{X}_{S_1}$ and $\mathbf{X}_{S_2}$ satisfies the assumption in Theorem \ref{Thm:CLTWeak}, and every component of $\mathbf{X}_F$ satisfies the assumption given at the end of Theorem \ref{Thm:Frac NCLT beta<0}.
\end{Thm}
The proof is similar to that of Theorem 3.5 of \citet{bai:taqqu:2013:multivariate}.  We only provide some heuristics. The processes $\mathbf{B}_2(t),\mathbf{Z}(t)$ and $\mathbf{Z}^{\mathbd{\beta}}(t)$ involve the same integrator $W(\cdot)$ because they are defined in terms of the same $\epsilon_i$'s. To understand the independence statement, note that the independence between $\mathbf{B}_2$ and $W$  stems from the uncorrelatedness between $X_{S_2}$ and $X_{S_1}$, since $X_{S_2}$ belongs to a discrete chaos of order $k\ge 2$, while $X_{S_1}$ belongs to a discrete chaos of order $k=1$. $\mathbf{B}
_2$ is therefore independent of $\mathbf{B}_1$. $\mathbf{B}_2$ is also independent of $\mathbf{Z}$ and $\mathbf{Z}^{\mathbd{\beta}}$, because $\mathbf{Z}$ and $\mathbf{Z}^{\mathbd{\beta}}$  have $W$ as integrators.

\begin{Rem}\label{Rem:indep}
The pairwise dependence between components of $\mathbf{Z}$, of $\mathbf{Z}^{\mathbd{\beta}}$, and between cross components in Theorem \ref{Thm:multi limit} can be  checked using the criterion due to \citet{ustunel:zakai:1989:independence}, that is, if $f\in L^2(\mathbb{R}^p)$ and $g\in L^2(\mathbb{R}^q)$, and both are symmetric, then the multiple Wiener-It\^o integrals $I_p(f)$ and $I_q(g)$ are independent, if and only if
\[
f \otimes_1 g(x_1,\ldots,x_{p+q-2}):=\int_{\mathbb{R}}f(x_1,\ldots,x_{p-1},y)g(x_{p},\ldots,x_{p+q-2},y) dy =0 ~~a.e..
\]
For example, suppose that  two generalized Hermite kernels $g_1$ and $g_2$ on $\mathbb{R}_+^{p}$ and $\mathbb{R}_+^{q}$ are symmetric, then the corresponding two generalized Hermite processes are independent if and only if
\begin{equation}\label{eq:indep crit}
\int_\mathbb{R}~ \int_{0}^t g_1(s-x_1,\ldots,s-x_{p-1},s-y)ds\int_0^t g_2(s-x_{p},\ldots,s-x_{p+q-2},s-y)ds~dy = 0 \qquad a.e. ~,
\end{equation}
where we use the  abbreviation $ g_j(\mathbf{x})=g_j(\mathbf{x})\mathrm{1}_{\{\mathbf{x}>\mathbf{0}\}}$, $j=1,2$.  Obviously, if $g_1$ and $g_2$ are both positive, then the dependence always holds. This is true, for example, for the symmetrized version of the kernels in (\ref{eq:nonsym Herm}).
\end{Rem}

\noindent\textbf{Acknowledgments.} This work was partially supported by the NSF grant DMS-1007616 and DMS-1309009 at Boston University.

\bibliographystyle{plainnat}

\begin{thebibliography}{32}
\providecommand{\natexlab}[1]{#1}
\providecommand{\url}[1]{\texttt{#1}}
\expandafter\ifx\csname urlstyle\endcsname\relax
  \providecommand{\doi}[1]{doi: #1}\else
  \providecommand{\doi}{doi: \begingroup \urlstyle{rm}\Url}\fi

\bibitem[Bai and Taqqu(2013)]{bai:taqqu:2013:multivariate}
S. Bai and M.S. Taqqu.
\newblock Multivariate limits of multilinear polynomial-form processes with
  long memory.
\newblock \emph{Statistics \& Probability Letters}, 83\penalty0 (11), 2473-2485 2013.

\bibitem[Bardet and Tudor(2010)]{bardet:tudor:2010:wavelet}
J-M. Bardet and C.A. Tudor.
\newblock A wavelet analysis of the {R}osenblatt process: chaos expansion and
  estimation of the self-similarity parameter.
\newblock \emph{Stochastic Processes and their Applications}, 120\penalty0
  (12):\penalty0 2331--2362, 2010.

\bibitem[Billingsley(1956)]{billingsley:1956:invariance}
P. Billingsley.
\newblock The invariance principle for dependent random variables.
\newblock \emph{Transactions of the American Mathematical Society}, 83\penalty0
  (1):\penalty0 250--268, 1956.

\bibitem[Dehling et~al.(2012)Dehling, Rooch, and Taqqu]{dehling:rooch:2012:non}
H. Dehling, A. Rooch, and M.S. Taqqu.
\newblock Non-parametric change-point tests for long-range dependent data.
\newblock \emph{Scandinavian Journal of Statistics}, 40(1):153-173, 2013.

\bibitem[Dobrushin(1979)]{dobrushin:1979:gaussian}
R.L.~Dobrushin.
\newblock {G}aussian and their subordinated self-similar random generalized
  fields.
\newblock \emph{The Annals of Probability},7(1):1-28, 1979.

\bibitem[Dobrushin and Major(1979)]{dobrushin:major:1979:non}
R.L. Dobrushin and P.~Major.
\newblock Non-central limit theorems for non-linear functional of {G}aussian
  fields.
\newblock \emph{Probability Theory and Related Fields}, 50\penalty0
  (1):\penalty0 27--52, 1979.

\bibitem[Embrechts and Maejima(2002)]{embrechts:maejima:2002:selfsimilar}
P.~Embrechts and M.~Maejima.
\newblock \emph{Selfsimilar Processes}.
\newblock Princeton University Press, 2002.

\bibitem[Giraitis et~al.(2012)Giraitis, Koul, and
  Surgailis]{giraitis:koul:surgailis:2009:large}
L.~Giraitis, H.L. Koul, and D.~Surgailis.
\newblock \emph{Large Sample Inference for Long Memory Processes}.
\newblock World Scientific Publishing Company Incorporated, 2012.

\bibitem[Gradshteyn and Ryzhik(2007)]{gradshteyn:2007:table}
I.S. Gradshteyn and I.M. Ryzhik.
\newblock \emph{Table of Integrals, Series, and Products}.
\newblock Table of Integrals, Series, and Products Series. Elsevier, 2007.


\bibitem[Janson(1997)]{janson:1997:gaussian}
S. Janson.
\newblock \emph{Gaussian Hilbert Spaces}. Cambridge Tracts in Mathematics, Volume 129.
\newblock Cambridge University Press, 1997.

\bibitem[Krakowiak and Szulga(1986)]{krakowiak:szulga:1986:random}
W. Krakowiak and J. Szulga.
\newblock Random multilinear forms.
\newblock \emph{The Annals of Probability}, 14(3):  955--973, 1986.

\bibitem[Lamperti(1962)]{lamperti:1962:semi}
J. Lamperti.
\newblock Semi-stable stochastic processes.
\newblock \emph{Transactions of the American Mathematical Society},
  104\penalty0 (1):\penalty0 62--78, 1962.

\bibitem[L{\'e}vy-Leduc et~al.(2011)L{\'e}vy-Leduc, Boistard, Moulines, Taqqu,
  and Reisen]{levy:boistard:taqqu:reisen:2011:asymptotic}
C.~L{\'e}vy-Leduc, H.~Boistard, E.~Moulines, M.S. Taqqu, and V.A. Reisen.
\newblock Asymptotic properties of {U}-processes under long-range dependence.
\newblock \emph{The Annals of Statistics}, 39\penalty0 (3):\penalty0
  1399--1426, 2011.

\bibitem[Maejima and Tudor(2012)]{maejima:tudor:2012:selfsimilar}
M.~Maejima and C.A.~Tudor.
\newblock Selfsimilar processes with stationary increments in the second
  {W}iener chaos.
\newblock \emph{Probability and Mathematical Statistics}, 32\penalty0 (1):167-186,
  2012.

\bibitem[Maejima and Tudor(2007)]{maejima:tudor:2007:wiener}
M. Maejima and C.A. Tudor.
\newblock {W}iener integrals with respect to the {H}ermite process and a
  non-central limit theorem.
\newblock \emph{Stochastic Analysis and Applications}, 25\penalty0
  (5):\penalty0 1043--1056, 2007.

\bibitem[Maejima and Tudor(2013)]{maejima:tudor:2013:distribution}
M. Maejima and C.A. Tudor.
\newblock On the distribution of the {R}osenblatt process. \emph{Statistics and Probability Letters}, 83(6):1490--1495
\newblock , 2013.

\bibitem[Major(1981{\natexlab{a}})]{major:1981:multiple}
P.~Major.
\newblock \emph{Multiple {W}iener-{I}t{\^o} Integrals}. Lecture Notes in Mathematics, Volume 849,
\newblock Springer, 1981{\natexlab{a}}.

\bibitem[Major(1981{\natexlab{b}})]{major:1981:limit}
P{\'e}ter Major.
\newblock Limit theorems for non-linear functionals of {G}aussian sequences.
\newblock \emph{Zeitschrift f{\"u}r Wahrscheinlichkeitstheorie und verwandte
  Gebiete}, 57\penalty0 (1):\penalty0 129--158, 1981{\natexlab{b}}.

\bibitem[Mandelbrot and Van~Ness(1968)]{mandelbrot:vanness:1968:fractional}
B.B. Mandelbrot and J.W. Van~Ness.
\newblock Fractional {B}rownian motions, fractional noises and applications.
\newblock \emph{SIAM Review}, 10\penalty0 (4):\penalty0 422--437, 1968.

\bibitem[Mori and Oodaira(1986)]{mori:toshio:1986:law}
T. Mori and H. Oodaira.
\newblock The law of the iterated logarithm for self-similar processes
  represented by multiple wiener integrals.
\newblock \emph{Probability theory and related fields}, 71\penalty0
  (3):\penalty0 367--391, 1986.

\bibitem[Nourdin(2012)]{nourdin:2012:selected}
I. Nourdin.
\newblock \emph{Selected aspects of fractional {B}rownian motion}.
\newblock Springer, 2012.

\bibitem[Nualart(2006)]{nualart:2006:malliavin}
D. Nualart.
\newblock \emph{The {M}alliavin Calculus and Related Topics}.
\newblock Springer-Verlag Berlin, Heidelberg, second edition, 2006.

\bibitem[Peccati and Taqqu(2011)]{peccati:taqqu:2011:wiener}
G.~Peccati and M.S. Taqqu.
\newblock \emph{Wiener Chaos: Moments, Cumulants and Diagrams: A survey with
  Computer Implementation}.
\newblock Springer Verlag, 2011.

\bibitem[Pipiras and Taqqu(2010)]{pipiras:taqqu:2010:regularization}
V.~Pipiras and M.S. Taqqu.
\newblock Regularization and integral representations of {H}ermite processes.
\newblock \emph{Statistics and Probability Letters}, 80\penalty0 (23):\penalty0
  2014--2023, 2010.

\bibitem[Rosenblatt(1979)]{rosenblatt:1979:some}
M.~Rosenblatt.
\newblock Some limit theorems for partial sums of quadratic forms in stationary
  {G}aussian variables.
\newblock \emph{Zeitschrift f{\"u}r Wahrscheinlichkeitstheorie und verwandte
  Gebiete}, 49\penalty0 (2):\penalty0 125--132, 1979.

\bibitem[Rosenblatt(1961)]{rosenblatt:1961:independence}
M. Rosenblatt.
\newblock Independence and dependence.
\newblock In \emph{Proc. Fourth Berkeley Symp. Math. Statist. Probab},
  Volume 2:431-443, 1961.

\bibitem[Samko et~al.(1993)Samko, Kilbas, and Marichev]{samko:1993:fractional}
S.G. Samko, A.A. Kilbas, and O.I. Marichev.
\newblock \emph{Fractional integrals and derivatives}.
\newblock Gordon and Breach Science Publishers, Yverdon, 1993.

\bibitem[Surgailis(1982)]{surgailis:1982:zones}
D.~Surgailis.
\newblock Zones of attraction of self-similar multiple integrals.
\newblock \emph{Lithuanian Mathematical Journal}, 22\penalty0(3):\penalty0
  327--340, 1982.

\bibitem[Taqqu(1979)]{taqqu:1979:convergence}
M.S. Taqqu.
\newblock Convergence of integrated processes of arbitrary {H}ermite rank.
\newblock \emph{Probability Theory and Related Fields}, 50\penalty0
  (1):\penalty0 53--83, 1979.

\bibitem[Tudor(2008)]{tudor:2008:analysis}
C.A. Tudor.
\newblock Analysis of the {R}osenblatt process.
\newblock \emph{ESAIM: Probability and Statistics}, 12\penalty0 (1):\penalty0
  230--257, 2008.

\bibitem[Ustunel and Zakai(1989)]{ustunel:zakai:1989:independence}
A.S. Ustunel and M.~Zakai.
\newblock On independence and conditioning on {W}iener space.
\newblock \emph{The Annals of Probability}, 17\penalty0 (4):\penalty0
  1441--1453, 1989.

\bibitem[Veillette and Taqqu(2012)]{veillette:taqqu:2012:properties}
M.S. Veillette and M.S. Taqqu.
\newblock Properties and numerical evaluation of the {R}osenblatt distribution. \emph{Bernoulli}, 19(3):982-1005,
\newblock  2013.

\end{thebibliography}

\bigskip

\noindent Shuyang Bai~~~~~~~ \textit{bsy9142@bu.edu}\\
Murad S. Taqqu ~~\textit{murad@bu.edu}\\
Department of Mathematics and Statistics\\
111 Cumminton Street\\
Boston, MA, 02215, US

\end{document}